\documentclass[11pt]{amsart}
\usepackage[foot]{amsaddr}   
\usepackage{graphicx, amssymb,amsmath, latexsym,cite,fancyhdr}
\usepackage[dvipsnames]{xcolor}
\usepackage{amsthm,arydshln}  
\usepackage{times}         
\usepackage[T1]{fontenc}    
\usepackage{url} 
\usepackage{mathabx,xfrac}

 \usepackage{geometry}
\usepackage{tikz-cd}
\usetikzlibrary{cd}

\newtheoremstyle{mythm} 
{6pt}%space above
{6pt}%space below
{\it}%body font 
{}%indent amount
{\bf}%theorem head font
{.}%punctation after theorem head
{.5em}%space after thm head
{}%thm head spec

\newtheoremstyle{mydef}
{6pt}%space above
{6pt}%space below
{}%body font
{}%indent amount
{\bf}%theorem head font
{.}%punctation after theorem head
{.5em}%space after thm head
{}%thm head spec

\newtheoremstyle{myrem}
{6pt}%space above
{6pt}%space below
{}%body font
{}%indent amount
{\bf}%theorem head font
{.}%punctation after theorem head
{.5em}%space after thm head
{}%thm head spec

\newtheoremstyle{myex}
{6pt}%space above
{3ex}%space below
{}%body font
{}%indent amount
{\bf}%theorem head font
{.}%punctation after theorem head
{.5em}%space after thm head
{}%thm head spec

\theoremstyle{mythm}
\newtheorem{theorem}{Theorem}[section]
\newtheorem{lemma}[theorem]{Lemma}
\newtheorem{proposition}[theorem]{Proposition}

\theoremstyle{mydef}
\newtheorem{definition}[theorem]{Definition}

\theoremstyle{myrem}
\newtheorem{remark}[theorem]{Remark}
\numberwithin{equation}{section}

\theoremstyle{myex}
\newtheorem{example}[theorem]{Example}

\newcommand{\Z}{\mathbb{Z}}
\newcommand{\Aut}{{\rm Aut}}

\newcommand{\rank}{{\rm rank}}

\def\gesy#1{\underline{\mathbf{#1}}}
\def\sz#1{\left|#1\right|}

\def\calA{{\mathcal A}}
\def\calR{{\mathcal R}}
\def\Oof{{\mathcal O}\/}
\def\F{{\mathbb F}}

\newcommand{\fox}[2]{\tfrac{\partial #1}{\partial #2}}
\newcommand{\foxb}[2]{\tfrac{\partial(#1)}{\partial #2}}  

\newcommand{\calm}{\mathcal{M}}
\newcommand{\calmhep}{\calm_{H,p,e}}

\DeclareMathOperator{\rad}{rad}

\newcommand{\ee}{{\text{\bf e}}} 

\newcounter{ithmcount}

\newenvironment{items}{
\begin{list}{$\alph{item})$}
{\labelwidth30pt \leftmargin30pt \topsep3pt \itemsep2pt \parsep0pt}}
{\end{list}}

\newenvironment{iprf}{\begin{list}{{\rm
	\alph{ithmcount})}}{\usecounter{ithmcount}\labelwidth-5pt
      \leftmargin0pt \topsep3pt \itemsep1pt \parsep2pt}}{\qedhere\end{list}}

\newenvironment{ithm}{\begin{list}{{\rm \alph{ithmcount})}}{\usecounter{ithmcount}\labelwidth18pt
      \leftmargin18pt \topsep3pt \itemsep1pt \parsep2pt}}{\end{list}}

\allowdisplaybreaks

\renewcommand{\leq}{\leqslant}
\renewcommand{\geq}{\geqslant}

\begin{document}

\vspace*{-1.8cm}

\title{Universal covers of finite groups}
  
\author[H.~Dietrich]{Heiko Dietrich}
\address{School of Mathematics, Monash University, VIC 3800, Australia}
\email{heiko.dietrich@monash.edu}
\author[A.~Hulpke]{Alexander Hulpke}
\address{Department of Mathematics, Colorado State University, Ft.~Collins, CO 80523, USA}
\email{hulpke@colostate.edu}

\begin{abstract}  
Motivated by the success of quotient algorithms, such as the well-known $p$-quotient or
solvable quotient algorithms, in computing information about finite groups, we describe how to compute finite extensions $\tilde
H$ of a finite group $H$ by a direct sum of isomorphic simple $\Z_p
H$-modules such that $H$ and $\tilde H$ have the same  number of generators.
Similar to other quotient algorithms, our description will be via a
suitable covering group of $H$. Defining this covering group requires a
study of the representation module, as introduced by Gasch\"utz in 1954. Our
investigation involves so-called Fox derivatives (coming from free
differential calculus) and, as a by-product, we prove that these can be
naturally described via a wreath product construction. An important application of our results is that they can be used to compute, for a given epimorphism $G\to H$ and
simple $\Z_p H$-module $V$, the largest quotient of $G$ that maps onto $H$
with kernel isomorphic to a direct sum of copies of $V$. For this we also provide a
description of how to compute second cohomology groups for the (not
necessarily solvable) group $H$, assuming a confluent rewriting system for
$H$. To represent the corresponding group extensions on a computer, we
introduce a new hybrid format that combines this rewriting system with  the polycyclic presentation of the module.
\end{abstract} 

\maketitle

\vspace*{-0.6cm}

%%%%%%%%%%%%%%
\section{Introduction}
\noindent There are three well-established ways to describe a group for a
computer: permutations, matrices, and presentations. A detailed account on how to compute with groups is given in  the books \cite{handbook,Seress,simsbook}.
Finite presentations, that is, a finite set of generators together with  a finite
set of relators, are often a natural and compact way to define groups. For groups given in this form, effective algorithms exist for special kinds of presentations (such as polycyclic
presentations) and certain tasks (such as computing abelian invariants).
In general, however, due to the undecidability of the word problem for groups (Novikov-Boone Theorem), many problems
have been shown to be algorithmically undecidable. What one can do, based on von Dyck's Theorem, is to  attempt to investigate such a group via its quotients. This is the idea of so-called \emph{quotient
algorithms}, and the main motivation of this paper. 

Let $G$ be a finitely presented group and let $\varphi\colon G\to H$ be an epimorphism onto a finite group. By the isomorphism theorem, $G/\ker\varphi\cong H$, so the structure of $H$ has implications for  $G$. For example, if $H$ is non-trivial, then this proves that  $G$ is non-trivial -- something which is in general undecidable for finitely presented groups. In practice, one attempts to find epimorphisms from $G$ onto groups $H$ that allow practical computations, for example, permutation or polycyclic groups.

The aim of quotient algorithms is to find (largest) quotients of $G$ with
certain properties. For example, the largest abelian quotient of $G$ is
$G/G'$, where $G'=[G,G]$ is the derived subgroup; the computation of $G/G'$
is straightforward via a Smith-Normal-Form calculation. The well-known
$p$-quotient algorithm of Macdonald \cite{macdonald}, Newman \& O'Brien \cite{pquotient}, and Havas \& Newman \cite{havas} attempts to construct, for a user-given prime $p$, the largest quotient of $G$
that is a  finite $p$-group, we refer to \cite[Section 9.4]{handbook} for a detailed
discussion and references; see also Remark \ref{remPGen} below. Often such a
largest quotient does not exist, so the algorithm takes as input a bound on
the nilpotency class of the $p$-quotient that one wants to construct. For a
discussion of other quotient algorithms we refer to \cite[Section
9.4.3]{handbook}. For example, using a similar approach as the $p$-quotient
algorithm, the nilpotent quotient algorithm of Nickel \cite{nickelnq} tries to compute
the largest nilpotent quotient of $G$. Solvable quotient algorithms, such as
described by Plesken~\cite{ples87}, Leedham-Green~\cite{leedham84},
and Niemeyer~\cite{niem94}, attempt to
construct solvable quotients of $G$ as iterated extensions; generalisations
to polycyclic quotients exist, see Lo \cite{lo}. For the case of non-solvable groups $H$, the $L^2$-quotient algorithm \cite{ples09}, and generalizations in~\cite{bridsonetal19}, find quotients
that are (close to) simple groups in particular classes, but these algorithms do not consider lifts to larger quotients.
The concept of lifting epimorphisms by a module, using a presentation for
the factor group to produce linear equations that yield 2-cocycles, is
suggested already in~\cite{ples87} and used for the case of
non-solvable groups in~\cite{holt89}. However, none of these algorithms provides a description
of an iterated lifting algorithm for arbitrary non-solvable quotients. Moreover, not all suggested approaches are available in general-purpose implementations (like the $p$-quotient algorithm is).  

We describe a new approach for non-solvable quotients that follows the iteration strategy used in some solvable quotient algorithms. Given an
epimorphism $\varphi\colon G\to H$ onto some finite group, we aim to extend it (if possible)  to a larger quotient of $G$ via an epimorphism
$\alpha\colon G\to K$ that satisfies $\ker\alpha\leq\ker\varphi$, that is, $\alpha$ factors through $\varphi$.  We assume that $\ker\varphi/\ker\alpha$ is a
finite semisimple module for $H$, so by an iteration we can discover \emph{any}
quotient of $G$ that is an extension of $H$ with a finite solvable subgroup. This approach assumes that the non-solvable part of the required quotient of $G$ has  been supplied as input, which mirrors the view of the \emph{solvable radical paradigm}, see \cite[Section~10.3]{handbook}.
This paradigm has been  used successfully in modern algorithms for
permutation or matrix groups, and relies on the fact that  every finite non-solvable group is an extension of 
a solvable normal subgroup (the \emph{radical}) with a
Fitting-free factor group (not affording any non-trivial solvable normal subgroup). We indicate in Section~\ref{secInitialQuot} how such an initial
epimorphism $\varphi$ can be found.

\subsection{Main results} In the following, $e$ is a positive integer and $p$ is a prime. We say a group is \emph{$e$-generated} if it can be generated by $e$ elements. Our first result is the following.

\begin{theorem}\label{thmHPE1}
Let $H$ be a finite $e$-generated group. There is a finite $e$-generated group $\hat H_{p,e}$, called the $p$-cover of $H$ of rank $e$, such that $\hat H_{p,e}$ is an extension of $H$ with an elementary abelian $p$-group, and  any other such $e$-generated extension of $H$ is a quotient of $\hat H_{p,e}$. 
\end{theorem} 

This result is proved in Theorem \ref{thmCover} based on a result of
Gasch\"utz. If $H$ is given as a finitely presented group, say $H=F/M$ with
$F$ free of rank $e$, then $\hat H_{p,e}$ can be defined as
$F/[M,M]M^{[p]}$, where $M^{[p]}$ denotes the subgroup of $M$ generated by all $p$-th powers. However, the Nielsen-Schreier Theorem shows that the
kernel of the projection $\hat H_{p,e}\to H$ is an elementary abelian
$p$-group of rank ${1+(e-1)|H|}$, which makes an explicit construction of
$\hat H_{p,e}$ as a finitely presented group, following this definition, infeasible in practice. In
Section \ref{seccoverconst} we therefore discuss an alternative description of $\hat H_{p,e}$,
using Fox derivatives, see Theorem \ref{thmPhi} for details.

While the
definition of $\hat H_{p,e}$ is straightforward, it is the new construction in Theorem \ref{thmPhi} that is our first main result. We do not explain it here, because this would require notation given in Section~\ref{seccoverconst}.

To  make our approach feasible in practice, we consider
a further reduction: We say a $\Z_p H$-module $A$ is \emph{$V$-homogeneous} if $A$ is a direct sum of finitely many copies of a simple $\Z_p H$-module $V$. In Section~\ref{secLiftEpi} we provide some references for the construction of simple $\Z_p H$-modules;  all modules we consider here are finite-dimensional.

\begin{theorem}\label{thmHVE1}
Let $H$ be a finite $e$-generated group with simple $\Z_pH$-module $V$. There is a finite $e$-generated group $\hat H_{V,e}$, called the $(V,e)$-cover of $H$, such that $\hat H_{V,e}$ is an extension of $H$ with a $V$-homogeneous module, and any other such $e$-generated extension of $H$ is a quotient of~$\hat H_{V,e}$.
\end{theorem}

In principle, one can construct $\hat H_{V,e}$ from $\hat H_{p,e}$, however, doing so would not resolve the issue that $\hat H_{p,e}$ is often too big in practice. Instead, we describe a direct construction. For this we show  that $\hat H_{V,e}$ is  a subdirect product of a split and a non-split
part, see Theorem \ref{thmhomogext}. 
The former part can be  obtained as a modification of 
our construction for $\hat H_{p,e}$; we discuss this in Proposition \ref{propT}. The latter part can be obtained by the cohomological methods described in
Section~\ref{secComp}; this requires that we have a confluent rewriting system for $H$.  Based on those two parts,  in Theorem~\ref{thmHve} we provide a practically
feasible construction of $\hat H_{V,e}$. To avoid technical details, the result
is formulated here as an existence statement, but the proof will be
constructive.

\begin{theorem}\label{thmHve1} 
Let $H$ be a finite, finitely presented, $e$-generated group and let $V$ be a simple $\Z_p H$-module. If a basis of $H^2(H,V)$ is known, there is an algorithm to construct $\hat H_{V,e}$, see Theorem~\ref{thmHve}.
\end{theorem}

Assuming a confluent rewriting system for $H$, we describe a construction algorithm for $H^2(H,V)$ in  Section~\ref{secComp}; this allows us to apply Theorem \ref{thmHve1} to construct $\hat H_{V,e}$.

Importantly, our results can be used for a non-solvable quotient algorithm. We discuss the details of the following theorem in Section \ref{secLiftEpi}. 
  
\begin{theorem}
 Let $\varphi\colon G\to H$ be an epimorphism from a finitely presented group onto a finite, finitely presented, $e$-generated group. Given a simple $\Z_p H$-module $V$ and a confluent rewriting system for $H$, there is an algorithm to construct an epimorphism $\alpha\colon G\to K$ where $\ker\alpha\leq \ker\varphi$  and $K$ is the largest $e$-generated quotient of $G$ that maps onto $H$ with $V$-homogeneous kernel.
\end{theorem}
  
Our last, and practically most relevant, result is a workable implementation
of our algorithms for the  computer algebra system {\sf GAP}~\cite{GAP}; we
discuss this in Section \ref{secPI}. Our code is available under \url{https://github.com/hulpke/hybrid}, and we aim to make it available as part of a standard GAP distribution. What makes our implementation
effective is a hybrid computer representation of the non-solvable extensions
of $H$ that combines confluent rewriting systems (for the non-solvable
factor) and polycyclic presentations (for the solvable normal subgroups); we
give details in Section~\ref{sechybgroup}.   We discuss some cost estimates
of our algorithm in Section \ref{secCost}.
Section~\ref{secEx} illustrates the scope of the algorithm in some examples. For instance, in Example~\ref{exHeineken} we have been able to compute, in a few minutes, an epimorphism from the infinite Heineken group $\mathcal{H}$ onto  $2^4.2^4.(2\times 2).2^4.2^4.2.(2\times 2^4).A_5$. This quotient had been constructed in a permutation representation of degree
138240 in~\cite{sinananholt} (with later work of Holt reducing to
permutation degree 15360), but our method works generically and avoids large
degree permutation representations.
 
\subsection{Comparison with other quotient algorithms}
We show in Remark~\ref{remPGen} that for groups $H$  of $p$-power order
our cover $\hat H_{V,e}$ is a generalisation of the $p$-covering group $H^\ast$,
and that our algorithm therefore generalises the  $p$-quotient
algorithm~\cite{pquotient}. For the case of a solvable $H$, several versions
of quotient algorithms have been proposed, for example
in~\cite{leedham84,ples87,niem94}.

The method of~\cite{leedham84,niem94}
constructs the maximal possible extension with a
module in a single step. When starting with an epimorphism $G\to H$ from a
free group $G$,  it will in fact construct the maximal cover $\hat H_{p,e}$.
This approach risks that in the process of forming this module (from
relations using vector enumeration) it will encounter a regular module of
$H$ (which often is infeasibly large) before reducing it back by further
relators. Our approach instead deliberately works with multiple covers, for
each of which its kernel is guaranteed to be much smaller than the regular
module.

While sharing many ideas with~\cite{ples87}, our approach differs in the following ways: first, we construct a universal cover and find the maximal possible lift of a given epimorphism $G\to H$ via a quotient of this cover; in \cite{ples87}, lifts are constructed in steps, each time extending by one copy of the module. Second, our construction of the cover reduces the extensions of $H$ that have to be determined using cohomology to a basis of the corresponding cohomology group, whereas the construction in~\cite{ples87} works with cosets of a subgroup of $H^2(H,M)$.

\subsection{Notation}
We denote by $e$  a positive integer and by $p$ a prime. We write $\Z_p$ for the integers modulo $p$. A group $G$ is an extension of $Q$ with $N$ if $G$ has a normal subgroup $M\cong N$ with $G/M\cong Q$; we usually identify $M=N$ and $G/N=Q$. A subgroup $U\leq A\times B$ of a direct product  is a subdirect product of
$A$ and $B$ if $U$ has surjective projections onto both $A$ and $B$.
In this case, \cite[Lemma 1.1]{thev97}
shows that for
$U_1=U\cap A$ and $U_2=U\cap B$ there is an isomorphism $\tau\colon A/U_1\to
B/U_2$, and $U$ is the preimage of $\{(aU_1,\tau(aU_1): a\in A\}$  under the natural projection from 
$A\times B$ to $A/U_1\times B/U_2$.

Throughout the paper, we use the following notation. We fix a prime $p$, a finitely presented group $G$, and an epimorphism $\varphi\colon G\to H$ onto a finite group $H$. Let $F$ be the free
group underlying the presentation of $G$ and denote its rank by $e$. Since $G$ is a quotient of
$F$ by a relation subgroup $R\unlhd F$, the epimorphism $\varphi$ lifts to a homomorphism
$\psi\colon F\to H$. Its kernel $M=\ker\psi$ will
map onto the kernel of any extension of $H$ that is a quotient for $G$. The situation is summarised by the following commutative diagram (whose first row is a short exact sequence).
\begin{equation*}
\begin{tikzcd}
  1 \arrow{r} & M \arrow{r} & F  \arrow{r}{\psi}\arrow{d} & H \arrow[r] & 1\\
          &          & G \arrow{ur}[swap]{\varphi} 
\end{tikzcd}
\end{equation*}

%%%%%%%%%%%%%%%%%%%%%%%%%%%%%%%%%%%%%%%%%%%%%%%%%%%%%%%%%%%%%%%%%%%%%%%%%%

\section{Definition of covers and the regular module}\label{secDefR}

\noindent  In this section
we define the covers $\hat H_{p,e}$ and $\hat H_{V,e}$ of $H$, and  recall some  results for the $p$-modular regular module $\Z_p H$. In later sections we investigate these covers and their construction in detail.

\subsection{The $p$-cover of rank $e$}
We start with a discussion of the so-called $p$-representation module of~$H$. We write
\[
M_p=[M,M]M^{[p]}
\]
for the smallest normal subgroup of $M$ whose corresponding quotient group is an elementary abelian $p$-group; here $M'=[M,M]$ is the derived subgroup of $M$ and $M^{[p]}$ is the subgroup generated by all $p$-th powers. The quotient $M/M_p$ is an $H$-module where $g\in H$ acts via
conjugation by any preimage under $\psi$; this action is well-defined since $M$
acts trivially on $M/M_p$ by conjugation. 
The Nielsen-Schreier Theorem
\cite[(6.1.1)]{rob82} shows that $M$ is free  of rank $s=1+(e-1)|H|$,
hence  $M/M_p$ is elementary abelian of rank $s$. Since  $M_p$ is characteristic in $M$, hence normal in $F$, one can form $F/M_p$. We show in Theorem~\ref{thmCover} that the isomorphism type of $F/M_p$ depends only on $H$, $p$ and $e$, but not on  $\psi$. In view of Theorem \ref{thmHPE1} (proved with Theorem \ref{thmCover}), this justifies the following definition:

\begin{definition}
\label{defcover}
We call \[\calmhep=M/M_p\quad\text{and}\quad \hat H_{p,e}=F/M_p\] the $p$-representation module of $H$ and the $p$-cover of $H$ of rank $e$, respectively.
\end{definition}
 
The structure of $\calmhep$ has been described by Gasch\"utz~\cite{gasch54},
see also the book of Gruenberg~\cite{gruen76} and papers
\cite{cossey80,griess}. Note that $\hat H_{p,e}$ is an extension of $H$ with
$\calmhep$; we present an explicit construction of $\hat H_{p,e}$ in
Section~\ref{seccoverconst}. However, the rank $s$ of the module $\calmhep$
is often
too large for practical calculations. To reduce the size of the cover, we
therefore restrict to the case of semisimple homogeneous modules, that is,
modules which are the direct sum of isomorphic copies of a simple module.
Doing so does not limit the scope of our techniques, because  any other
extension with a module can be considered as an iterated extension with
semisimple homogeneous modules.

\subsection{The $(V,e)$-cover.} 

\begin{definition}\label{defHomCompNew} 
  Let $V$ be a simple $\Z_p H$-module. For a  $\Z_p H$-module $A$ let $V(A)$ be the smallest
submodule of $A$ such that $A/V(A)$ is $V$-homogeneous. The $(V,e)$-cover of $H$ is 
\[\hat H_{V,e}=\hat H_{p,e}/V(\calmhep);\]
by construction, it is the largest $e$-generated group that maps onto $H$ with 
$V$-homogeneous kernel.
\end{definition}

Recall that the radical $\rad(A)$ of an $H$-module $A$ is the intersection of all maximal submodules, and
$\rad(A)=0$ if no such submodules exist.
The following lemma seems well-known, see e.g.\ \cite[Introduction, \S 5]{curtisreiner}, but we could not find a
reference that includes all statements concisely in one place; therefore we include a short proof
in Appendix \ref{app} for completeness. It follows that  $\rad(A)\le V(A)$,
and therefore the structure of $A/V(A)$ is determined by the radical factor
of $A$.

\begin{lemma}\label{lemRad}
  Let $A$ and $B$ be   $H$-modules; let $C\leq A$ be a submodule. 
  \begin{ithm}
  \item We have $\rad(C)\leq \rad(A)$ and $\rad(A\oplus B)=\rad(A)\oplus\rad(B)$.
  \item If $\sigma\colon A\to B$ is an $H$-module homomorphism, then $\sigma(\rad(A))\leq \rad(B)$.
  \item We have $\rad(A/C)=(\rad(A)+C)/C$, and $A/C$ is semisimple if and only if $\rad(A)\leq C$.
  \end{ithm}
\end{lemma}

A practically
feasible construction of $\hat H_{V,e}$ is discussed in Section~\ref{secHve}.  Here we conclude with a comment on the $p$-cover in the $p$-quotient algorithm.

\begin{remark}\label{remPGen} 
If $H$ is a finite $p$-group, then it is natural to compare $\hat H_{p,e}$
with the $p$-cover of $H$ as defined in the $p$-quotient algorithm, see
\cite[Section 9.4]{handbook} for proofs and background information.  If $H$
has rank $e$ (that is, every minimal generating set of $H$ has size $e$),
then its $p$-cover $H^\ast$ is an $e$-generated extension of $H$ with a central elementary abelian $p$-group $N$, and every other such extension of $H$ is a quotient of $H^\ast$, see \cite[Theorem~9.18]{handbook}. The group $H^\ast$ is unique up to
isomorphism, and if $H=F/M$ with $F$ a free group of rank  $e$, then $H^\ast\cong F/[F,M]M^{[p]}$. In particular, $H^\ast$ is a quotient of
$\hat H_{p,e}$. Since $N$ is the direct sum of copies of the 1-dimensional
trivial $\Z_p H$-module $\pmb{1}$, it follows that  $H^\ast\cong \hat
H_{\pmb{1},\rank(H)}$ is a special case of our $p$-cover $\hat H_{V,e}$
\end{remark}

%%%%%%%%%%%%%%%%%%%%%%%%%%%%%%%%%%%%%%%%%%%%%%%%%%%%%%%%%%%%%%%%%%%%%%%%%%%%%%%%%%%%%%%%%%

\subsection{The structure of the regular module}\label{secRegMod}
We recall the following results for the regular module $\mathbb{F}H$ where $H$ is a finite group and $\mathbb{F}$ is a finite field.  Following \cite[Definition~1.5.8]{luxpahlings}, we call an extension
field $\underline{\mathbb{F}}\geq \mathbb{F}$ a \emph{splitting field} for an $\mathbb{F}$-algebra $A$,
if every simple $\underline{\mathbb{F}}A$-module is absolutely simple. It is proved in \cite[Lemma~1.5.9]{luxpahlings} that if $\dim_\mathbb{F} A<\infty$, then there exists a splitting field $\underline{\mathbb{F}}$ such that the extension $\underline{\mathbb{F}}>\mathbb{F}$ has finite degree. This allows us to state the following lemma and theorem, which are consequences of standard results of modular representation theory; due to their importance for this work, proofs of both results are contained in Appendix~\ref{app}.

\begin{lemma}
\label{lemKH}Let $\mathbb{F}$ be a finite field and let $\underline{\mathbb{F}}$ be a finite degree splitting field for $\mathbb{F}H$. For an $\mathbb{F}H$-module $V$ let $\underline{\mathbb{F}}V=\underline{\mathbb{F}}\otimes_\mathbb{F} V$ be the $\underline{\mathbb{F}}H$-module arising from $V$  by extending  scalars.
\begin{ithm}
\item If $V$ is a simple $\mathbb{F}H$-module, then $\underline{\mathbb{F}}V$ is a direct sum
of non-isomorphic simple $\underline{\mathbb{F}}H$-modules.
\item We have  $\underline{\mathbb{F}}\rad(\mathbb{F}H)=\rad(\underline{\mathbb{F}}H)$.
\end{ithm}
\end{lemma}

\begin{theorem}\label{structureKH}
Let $H$ be a finite group.
If $\mathbb{F}$ is a field in finite characteristic,
then the regular module can be decomposed as
\begin{eqnarray}\label{eqDecPj}\mathbb{F}H=D_1^{r_1}\oplus\ldots\oplus D_t^{r_t},
\end{eqnarray}
where each  $D_i$ is a module that is  indecomposable and  projective (as a direct summand of the free module). The factors $D_i/\rad(D_i)$ are simple,  mutually non-isomorphic, and $t$ is the number of isomorphism types of simple $\mathbb{F}H$-modules. The isomorphism type of each $D_j$ is determined uniquely by the isomorphism type of $D_j/\rad(D_j)$, and we have 
\[\mathbb{F}H/\rad(\mathbb{F}H)=\bigoplus\nolimits_{i=1}^t (D_i/\rad(D_i))^{r_i}.\]
Each multiplicity $r_i$ is the dimension of an
absolutely simple constituent of $D_i/\rad(D_i)$; if  $\mathbb{F}$ is of
sufficiently large degree over the prime field 
or if $\F$ is algebraically closed, then  $r_i=\dim (D_i/\rad(D_i))$. 
\end{theorem}

We are particularly interested in the regular module $\Z_p H$; we fix the following notation for the remainder of this paper.

\begin{definition}\label{defRhp}
We write $R_{H,p}$ for the $p$-modular regular $H$-module, that is, $R_{H,p}\cong \Z_p H\cong \Z_p^m$ as $H$-modules, where $|H|=m$. Applying \eqref{eqDecPj}, we decompose
\begin{eqnarray}\label{eqDecRhp}
R_{H,p}=D_1^{r_1}\oplus\ldots\oplus D_t^{r_t}.
\end{eqnarray}
Writing  $E_i=D_i/\rad(D_i)$, the set $\{E_1,\ldots,E_t\}$ forms a complete set of representatives of simple $\Z_p
H$-modules;  we assume  $E_1=\pmb{1}$ is the 1-dimensional trivial module. Each $r_i=\dim_{\Z_p} C_i$, where $C_i$ is an absolutely simple
constituent of $E_i$ over the algebraic closure of $\Z_p$.
\end{definition}

%%%%%%%%%%%%%%%%%%%%%%%%%%%%%%%%%%%%%%%%%%%%%%%%%%%%%%%%%%%%%%%%%%%%%%%%%%%%%%%%%%%%%%%%%%%%%%%%%%%

\section{Uniqueness of the cover and a construction}
\label{seccoveruniq}

\noindent Recall that $\psi\colon F\to H$ has kernel $M$ and that $\hat H_{p,e}=F/M_p$ is an extension of $H$ with the elementary abelian module $\calmhep=M/M_p$. The following lemma, due to Gasch\"utz \cite{gasch55}, shows that $\psi$ factors through any $e$-generated extension of $H$ with an elementary abelian $p$-group.

\begin{lemma}{\rm (\!\cite[Satz 1]{gasch55})}
\label{lemMAN}
Let $N\unlhd K$ be a finite normal subgroup of an $e$-generated group $K$. If $K/N$ is generated by $\{k_1N,\ldots, k_e N\}$, then there are $n_1,\ldots,n_e\in N$ with $K=\langle k_1n_1,\ldots,k_en_e\rangle$.
\end{lemma}

The next result  proves Theorem \ref{thmHPE1} and shows that the cover $\hat H_{p,e}$ is independent of the chosen projection $\psi\colon F\to H$; this theorem is largely a corollary to a result of Gasch\"utz \cite{gasch54}. Similar universal properties hold for covers of other quotient algorithms, cf.\ Remark \ref{remPGen} for the $p$-cover.

\begin{theorem}\label{thmCover}
The group $\hat H_{p,e}$ is an $e$-generated extension of $H$ with an elementary abelian $p$-group, and every other such extension of $H$ is a quotient of $\hat H_{p,e}$. The isomorphism type of $\hat H_{p,e}$ depends only on $H$, $p$, and $e$; the same holds for the $H$-module structure of $\calmhep$.
\end{theorem}
\begin{proof}
The first claim on $\hat H_{p,e}$ follows by construction. Now consider an  $e$-generated group $L$ with epimorphism $\tau\colon L\to H$ and $Y=\ker \tau$ an elementary abelian $p$-group. By Lemma \ref{lemMAN}, we can lift any generating set of
$H$ of size $e$ to a generating set of $L$; since $F$ is free, we can
therefore factor $\psi$ through $L$, that is, there is a homomorphism
$\beta\colon F\to L$ such that $\tau\circ\beta =\psi$.  Since $\beta(M)\leq
\ker \tau$ is elementary abelian of exponent $p$, we have
$\beta(M_p)=\beta(M'M^{[p]})=1$. This proves that $\beta$ induces an epimorphism from $\hat H_{p,e}$ to $L$, as required. To prove uniqueness of $\hat H_{p,e}$, consider  an $e$-generated group $K$ with the same properties as stipulated for $\hat H_{p,e}$.  By assumption, there exist epimorphisms $\hat H_{p,e}\to K$ and $K\to \hat H_{p,e}$; since both groups are finite,  $\hat H_{p,e}\cong K$.  That the isomorphism type of $\calmhep$ as $H$-module is independent from $\psi$ follows from \cite[Satz 1]{gasch54}.
\end{proof} 
 
Later we require  the following result about the structure of $\calmhep$:

\begin{theorem}{\rm (\!\cite[Satz 2 \& 3 \& 5 \& 6]{gasch54})}\label{thmMain}
Let $H$ be a finite $e$-generated group . The $\Z_p
H$-modules  $\calmhep$ and $(R_{H,p})^{e-1}\oplus\pmb{1}$ have
the same multiset of simple composition factors. Furthermore,  
$\calmhep\cong \mathcal{A}\oplus \mathcal{B}$ as  $H$-modules, where
$\mathcal{A}$ is a direct summand of $(R_{H,p})^e$, and so  a projective
module, and if $N\unlhd \hat H_{p,e}$ such that $N\leq \calmhep$ and $\hat H_{p,e}/N$ splits
over $\calmhep/N$, then $\mathcal{B}\le N$. 
\end{theorem}

\begin{remark}\label{remAB}
A detailed description of $\mathcal{A}$ and $\mathcal{B}$ is given in \cite{gasch54}.  In the following we use the notation of Definition \ref{defRhp}. If $p$ divides $|H|$, then $\rad(D_1)\ne 0$ and we
define integers $s_1,\ldots,s_t$  by \[\rad(D_1)/\rad(\rad(D_1))=E_1^{s_1}\oplus\ldots\oplus E_t^{s_t}.\]Now $S=D_1^{s_1}\oplus\ldots\oplus D_t^{s_t}$  is the projective cover of $\rad(D_1)$, cf.\ \cite[p.~256]{griess}, and $\mathcal{B}$ is defined as the kernel of the projection $S\to\rad(D_1)$. As shown in \cite[Satz 5']{gasch54} and \cite[p.~256--258]{griess}, this
kernel is unique up to isomorphism and does not contain
a direct summand isomorphic to any $D_1,\ldots,D_t$. If $p\nmid |H|$, then $\mathcal{B}=0$ and each $s_i=0$. We have $\mathcal{A}=D_1^{e-s_1}\oplus D_2^{(e-1)r_2-s_2}\oplus\ldots\oplus D_t^{(e-1)r_t-s_t}$.
\end{remark}

\subsection{A construction of the $p$-cover}
\label{seccoverconst}

\noindent
The definition of $\hat H_{p,e}$ as $F/M_p$ offers a way of constructing it as a finitely presented group.
However, the large rank of the module $\calmhep$ makes this infeasible
in all but the smallest examples. In this section we explore a different way and describe the
cover via so-called Fox derivatives and a wreath product construction. 

%%%%%%%%%%%%%%%%%%%%%%%%%%%
\subsection{Fox derivatives}
We first recall some results from \cite[Section 11.4]{john97}.  Let $F$ be
free  on the set $X=\{x_1,\ldots,x_e\}$. Since we will be working in the group ring
$\Z F$,  we denote the identity in $F$ (and in its quotient groups) by $\ee$
to avoid confusion with the unit $1\in\Z$. 

The Fox derivative of $x\in X$ is defined as the unique map
\[
\fox{}{x}\colon F\to \Z F
\]
that maps $x$ to $\ee$ and all other generators to zero, and 
satisfies the \emph{Leibniz' rule}
\[
\foxb{uv}{x} = (\fox{u}{x})v+\fox{v}{x}
\] 
for all $u,v\in F$. By abuse of notation, we also denote by $\fox{}{x}$ its linear extension to $\Z F$.
 
\begin{remark}\label{remFox}
The Leibniz' rule yields that
\[
\fox{\ee}{x}=0\quad\text{and}\quad \foxb{s^{-1}}{x}=-\fox{s}{x}s^{-1}.
\] 
The image of $w\in F$ under $\fox{}{x}$ is a sum of terms, one for each
occurrence of $x^{\pm 1}$ in $w$: the term corresponding to $w=axb$ is $b$,
and the term corresponding to $w=ux^{-1}v$ is $-x^{-1}v$. For example, if
$w=axbx^{-1}c$ where $a,b,c\in F$ do not contain $x^{\pm 1}$, then
$\foxb{w}{x}=bx^{-1}c-x^{-1}c$.
\end{remark}

By abuse of notation, we identify the projection $\psi\colon F\to H$ with
the induced homomorphism
\begin{eqnarray}\label{eqpsi}
\psi\colon (\Z F)^e\to (\Z H)^e,
\end{eqnarray}
and combine the Fox derivatives to a map 
\[
\partial\colon F\to (\Z F)^e,\quad w\mapsto (\fox{w}{x_1},\ldots,\fox{w}{x_e}).
\]
The composition of these maps gives $\psi\circ\partial\colon F\to (\Z H)^e$. The main result on Fox derivatives required in this work is \cite[Proposition 5]{john97}, which states that
\begin{eqnarray}\label{eqFox}\ker(\psi\circ\partial)=M'.
\end{eqnarray}In the next section we will use this fact to describe a group isomorphic to $\hat H_{p,e}$.

\subsection{A wreath product construction}\label{secPsip}
To remain within the class of groups, we identify the group ring $\mathbb{Z}H$ with a subgroup of
the regular wreath product $\mathbb{Z}\wr H$.
Suppose we have $|H|=m$, and consider 
\[
W= \Z\wr H = H \ltimes \Z^m,
\]
where the $m$ copies of $\Z$ in $\Z^m$ are labeled by the elements of $H$.
We
write $0=(0,\ldots,0)\in \Z^m$ and, if $h\in H$ and $z\in \Z$, then
\[
z(h)\in \Z^m\leq W
\]
denotes the element of $\Z^m$ with $z$ in position labeled $h$, and $0$s
elsewhere. Thus,  if $a,b,g,h\in H$, then $(a,1(g)),(b,1(h))\in W$ satisfy
\[
(a,1(g))\cdot(b,1(h))=(ab,1(gb)+1(h))\quad\text{and}\quad
(a,1(\ee))^{-1}=(a^{-1},-1(a^{-1})).
\]
For each $i\in \{1,\ldots,e\}$ define the homomorphism $\psi_i\colon F\to W$ by
\[
\psi_i\colon F\to W,\quad \psi_i(x_j)=\begin{cases} (\psi(x_j),0)&\text{if
$i\ne j$}\\ (\psi(x_j),1(\ee))&\text{if $i=j$.} \end{cases}
\]
We now prove that $\psi_i$ is closely related to the Fox derivative
$\fox{}{x_i}$. For this we  identify $\Z H$ with $\Z^m$ via the
additive isomorphism $\Z H\to \Z^m$ that maps each $g\in H$ to
$1(g)\in \Z^m$; this can be used to let $\psi\colon (\Z F)^e\to (\Z H)^e$ in \eqref{eqpsi} induce a homomorphism
\[\zeta\colon \Z F
\to \Z^m.\]

\begin{proposition}\label{propFoxWr}  
 If $i\in\{1,\ldots, e\}$ and $w\in F$, then
\[\psi_i(w)=(\psi(w),\zeta(\fox{w}{x_i})),\]
and  $\zeta(\fox{w}{x_i})=0$ if and only if $\psi(\fox{w}{x_i})=0$.
\end{proposition}
\begin{proof}
 For simplicity, write $\tau=\psi_i$ and $x=x_i$. Write  $w=w_1x^{\varepsilon_1}w_2x^{\varepsilon_2}\ldots w_k x^{\varepsilon_k}w_{k+1}$ where each $\varepsilon_j\in\{\pm1\}$ and each $w_j\in F$ is reduced and does not contain $x^{\pm 1}$. We prove the claim by induction on $k$. If $k=0$, then $w=w_1$ and $\tau(w)=(\psi(w),0)=(\psi(w),\fox{w}{x})$. For $k=1$ we have $w=w_1x^{\varepsilon_1}w_2$, which requires a case distinction: if $\varepsilon_1=1$, then
  \begin{eqnarray*}
    \tau(w)&=&(\psi(w_1),0)\cdot(\psi(x),1(\ee))\cdot(\psi(w_2),0)=(\psi(w),1(\psi(w_2)))=(\psi(w),\zeta(\fox{w}{x}));
  \end{eqnarray*}
if  $\varepsilon=-1$, then
  \begin{eqnarray*}
    \tau(w)&=&(\psi(w_1),0)\cdot(\psi(x)^{-1},-1(\psi(x)^{-1}))\cdot(\psi(w_2),0)\\&=&(\psi(w),-1(\psi(x^{-1}w_2)))\\&=&(\psi(w),\zeta(\fox{w}{x})).
  \end{eqnarray*}
  
  Now let $k\geq 2$ and write $w=w'x^{\varepsilon_k}w_{k+1}$; by the  induction hypothesis, we have
  \begin{eqnarray*}
    \tau(w)&=&\tau(w')\tau(x^{\varepsilon_k}w_{k+1})=(\psi(w'),\zeta(\fox{w'}{x}))\cdot (\psi(x^{\varepsilon_k}w_{k+1}),\zeta(\fox{x^{\varepsilon_k}w_{k+1}}{x}))=(\psi(w),\fox{w}{x}),
  \end{eqnarray*}
where
the last equation follows from the Leibniz' rule.
\end{proof} 

Let $W_{(p)}=H\ltimes \Z_p^m$ be the $p$-modular version of
$W$. We now combine $\psi_1,\ldots,\psi_e$ to
\[
\Psi=\psi_1\times\ldots\times\psi_e\colon F\to W^e.
\]
 and, induced by the natural projection $\Z\to\Z_p$, define $\Psi_p\colon F\to (W_{(p)})^e$ via \begin{eqnarray}\label{eqPsip}\Psi_p\colon F\stackrel{\Psi}{\to} W^e \stackrel{\rm proj}{\to} (W_{(p)})^e.
 \end{eqnarray} The homomorphism  $\Psi_p$ can be used to construct the $p$-cover $\hat H_{p,e}$ and the module $\calmhep$. 

\begin{theorem}\label{thmPhi} With the previous notation and Definition \ref{defcover}, the following hold.
  \begin{ithm}
  \item We have $\ker \Psi=M'$ and $\ker \Psi_p=M_p$.
  \item The $p$-cover $\hat H_{p,e}$ of $H$ of rank $e$ is isomorphic to $\Psi_p(F)$, and $\calmhep\cong \Psi_p(M)$ as $H$-modules. 
  \end{ithm}
\end{theorem}
\begin{proof}
  \begin{iprf}
  \item By Proposition \ref{propFoxWr} we have  $w\in \ker\Psi$ if and only if $\psi(w)=\ee$ and $\zeta(\fox{w}{x_i})=0$ for every~$i$, if and only if $w\in M$ and $\psi(\fox{w}{x_i})=0$ for every $i$, if and only if $w\in M$ and $\psi\circ\partial(w)=0$, if and only if $w\in M'$, see \eqref{eqFox}. It follows from this that $M/M'\cong \Psi(M)\leq \Z^{me}$, in particular, $\Psi_p$ induces a map $M/M'\to \Z^{me}\to \Z^{me}_p$ whose kernel is the preimage of $(p\Z)^{me}$ under $\Psi|_{M/M'}$, which is $M^{[p]}M'/M'$. In conclusion, $\ker\Psi_p=M_p$, as claimed.
  \item It  follows from a) and Theorem \ref{thmCover} that  $\Psi_p(F)\cong F/M_p\cong \hat H_{p,e}$. By the isomorphism theorem, $\Psi_p$ yields an isomorphism $\alpha\colon M/M_p\to \Psi_p(M)$, $rM_p\mapsto \Psi_p(r)$. Let $r\in M$, write $g\in H=F/M$ as $g=fM$, and note $(rM_p)^g=r^fM_p$. Since  $\Psi_p(M)\leq \Z^{me}\leq W^e$ is abelian, it follows that the conjugation action of $\Psi_p(f)$ on $\Psi_p(M)$ is conjugation by  $\psi(f)=g$.  Now $\alpha((rM_p)^g)=\Psi_p(r^f)=\Psi_p(r)^{\Psi_p(f)}=\Psi_p(r)^{g}=\alpha(rM_p)^g$ shows that $\alpha$ is an $H$-module isomorphism.
  \end{iprf}
\end{proof}

The construction in Theorem \ref{thmPhi} uses a wreath product with $|H|=m$ factors $\mathbb{Z}_p$; this makes it practical only for reasonably small groups $H$.

%%%%%%%%%%%%%%%%%%%%%%%%%%%%%%%%%%%%%%%%%%%%%%%%%%%%%%%%%%%%%%%%

\section{Extensions with homogeneous modules}
\label{secExtHom}

\noindent Let $V$ be a simple $\Z_p H$-module. In this section we study the
structure of extensions $E$ of  $H$ with a $V$-homogeneous module
$K\trianglelefteq E$. We will apply this later to the construction of the  cover $\hat H_{V,e}$, but the analysis
applies to any such extension $E$.

Since $K$ is $V$-homogeneous, any simple quotient module of $K$ will be
isomorphic to $V$, the intersection of the maximal $H$-submodules of
$K$ is trivial,  and submodules of $K$ correspond to normal subgroups of $E$ contained in $K$. This implies that $E$ is a subdirect product of extensions of $H$ with~$V$. We
(naturally) assume that in each of these extensions the projection onto $H$
is  induced by the projection $E\to H$, which allows us to simply identify these
factors in the subdirect product. 

To fix notation, we recall the basic setup of extension theory~\cite[Section
  11]{rob82}.

\begin{definition}Every extension of  $H$ with $V$ is isomorphic to a group
$E_\gamma$ with underlying element
set $H\times V$ and multiplication
\begin{equation} 
(g,v)\cdot(h,w)=\left(gh,v^hw\gamma(g,h)\right)
\label{extensionmult}
\end{equation}
for a 2-cocycle $\gamma\in Z^2(H,V)$. Note that we write $V$ multiplicatively, but we
consider $Z^2(H,V)$ and $H^2(H,V)$ as additive groups.
We call $E_\gamma$ the extension
corresponding to $\gamma$ and call the map
\[
\varepsilon_{\gamma}\colon E_\gamma\to H,\quad (h,v)\mapsto h,
\]
its natural epimorphism. Non-split extensions correspond to cocycles in
$Z^2(H,V)$ that lie outside the subgroup of 2-coboundaries $B^2(H,V)$. 
\end{definition}

We first study the interplay between extensions and subdirect
products.

\begin{lemma}  
\label{lemsubspl}
Let $E_1,\ldots,E_n$ be extensions of $H$ with $H$-modules $V_1,\ldots,V_n$, respectively, and let  $E$ be the subdirect product of the $E_i$, defined by identifying the factor groups isomorphic to $H$; let $K\unlhd E$ be the kernel of the projection $E\to H$.
\begin{ithm}
\item If each $E_i$ is split over $V_i$, then $E$ is split over $K$.
\item There exists a unique normal subgroup $L\unlhd E$ that is minimal with
respect to $E/L$ being split over $K/L$. In particular, $E$ is a subdirect product of non-split extensions of $H$ with the split extension $E/L$
of $H$. Every quotient of $E$ that is a split extension of $H$ is a quotient of $E/L$.
\end{ithm}
\end{lemma}
\begin{proof}
\begin{iprf}
\item It is sufficient to prove this for $n=2$. We can assume that the underlying set of $E$ is $H\times V_1\times V_2$ and  $K=\{(1,v_1,v_2)\mid v_i\in V_i\}$. If $\{(h,1): h\in H\}$ is a complement to $V$ in each $E_i$, then  $\{(h,1,1)\mid h\in H\}$ is a complement to $K$ in $E$.
\item Let $\mathcal{N}$ be the collection of all $N\unlhd E$ with $N\leq K$ such that $E/N$ is split over $K/N$. Note that the homomorphism $E\to \prod_{N\in\mathcal{N}} E/N$, $e\mapsto \prod_{N\in\mathcal{N}} eN$ has kernel $L=\bigcap_{N\in\mathcal{N}} N$ and its image is a subdirect product of all $E/N$, defined by identifying the factor groups isomorphic to $H$.  Since each such $E/N$ splits, part a) shows that $E/L$ is split over $K/L$. It follows that $E$ is the subdirect product of $E/L$ with those $E_i$ that are non-split. If $Q$ is a quotient of $E$  that is a split extension of $H$, then $Q\cong E/M$ for some $M\in\mathcal{N}$; this implies the last claim.
\end{iprf}
\end{proof}

\begin{definition}\label{defSK}
The subgroup $L$ in Lemma \ref{lemsubspl}b) is called the
{\em split kernel} of the extension $E$.
\end{definition}

We now show that subdirect products of extensions behave well under cocycle
arithmetic.

\begin{lemma}\label{lemSDP}
Let $V$ be a simple $\Z_p H$-module and  $\beta,\gamma\in Z^2(H,V)$.
\begin{ithm}
  \item Let $E$ be the subdirect product of $E_\beta$ and $E_\gamma$ defined by identifying $\varepsilon_\beta(E_\beta)=\varepsilon_\gamma(E_\gamma)$.
Let $\zeta=\beta+\gamma$.
There exists $N\unlhd E$ such that $E/N\cong E_\zeta$ and 
$N\cap\ker \varepsilon_\beta=1$. In particular, $E$ is isomorphic to the
subdirect product of $E_\beta$ and $E_\zeta$, defined by identifying
$\varepsilon_\beta(E_\beta)=\varepsilon_\zeta(E_\zeta)$.

\item The statement of a) holds for $\zeta=r\beta+\gamma$ with arbitrary $r\in\Z_p$.

\item
Let $D$ be a group with epimorphism $\pi\colon D\to E_\beta$. Let $E$ be the subdirect product of $D$ with $E_\gamma$ defined by identifying  $\varepsilon_\beta(\pi(D))=\varepsilon_\gamma(E_\gamma)$.
For every $\zeta=r\beta+\gamma$ with $r\in \Z_p$, the group $E$ is
isomorphic to the subdirect product of $D$ with $E_\zeta$, defined by identifying  $\varepsilon_\beta(\pi(D))=\varepsilon_\zeta(E_\zeta)$.

\end{ithm}
\end{lemma}
\begin{proof}
  \begin{iprf}
\item Up to isomorphism, we can identify $E$ with the Cartesian product 
$H\times V\times V$ with multiplication
\[
(a,v,w)\cdot (b,x,y)=
\left(ab,\, v^b x \beta(a,b),\,w^b y \gamma(a,b)\right)
\]
and natural projections  $\tau\colon E\to E_\beta$, $(a,v,w)\mapsto (a,v)$, and
$\sigma\colon E\to E_\gamma$, $(a,v,w)\mapsto (a,w)$. Let
\[
K=(\ker\tau)(\ker\sigma)=1\times V\times V \quad\text{and}\quad N=\left\{(1,v,v^{-1}): v\in V\right\}\leq K;
\]
note that $K,N\unlhd E$; the latter holds, since the $(a,1,1)$-conjugate of
$(1,v,v^{-1})$ is $(1,v^a,(v^a)^{-1})$. Furthermore $K/N\cong V$ as
$H$-modules. Now consider the natural homomorphism $\nu\colon E\to E/N$;
note that every element in $E/N$ has the form $(a,v,1)N$, and $\nu$ maps
$(a,v,w)$ to $(a,vw,1)N$. In particular,  the multiplication in $E/N$ is
\[
(a,v,1)N\cdot(b,w,1)N = (ab,v^bw\beta(a,b),\gamma(a,b))N=(ab,v^bw\beta(a,b)\gamma(a,b),1)N,
\]
which proves that $(a,v,1)N\to (a,v)$ defines an isomorphism $E/N\cong
E_\zeta$ where $\zeta=\beta+\gamma$. By abuse of notation, we consider the
epimorphism $\nu\colon E\to E_\zeta$, $(a,v,w)\mapsto (a,vw)$. Since the
homomorphism $\tau\times\nu\colon E\to E_\beta\times E_\zeta$ is injective,
the claim follows. 

\item This follows by an iterative application of a).

\item
Write
$A=\ker\pi$ and let $A\leq B\leq D$ such that $D/A\cong E_\beta$ and $B/A$ is the kernel of $\varepsilon_\beta\colon E_\beta\to H$, so $D/B\cong H$. As done in a), we identify $E$ with the Cartesian product
$H\times B\times V$ and note that
  \begin{eqnarray*}
    \tilde D&=&\{(h,b,1): h\in H,b\in B\}\cong D\quad\text{and}\\
    \tilde E_\gamma&=& \{(h,1,v): h\in H,v\in V\}\cong E_\gamma,
  \end{eqnarray*}
with corresponding natural projections  $\pi_{\tilde D}\colon E\to \tilde
D$, $(h,b,v)\mapsto (h,b,1)$, and $\pi_{\tilde E_\gamma}\colon E\to\tilde
E_\gamma$, $(h,b,v)\mapsto (h,1,v)$.   Note that  $L=\{(1,a,1): a\in A\}$ is
normal in $\tilde D$, and $\tilde D/L\cong E_\beta$. In particular, $L\unlhd
E$, and $E/L$ is isomorphic to the subdirect product of $E_\beta$ and
$E_\gamma$ defined by identifying the common quotient $H$. By b), there
exists  $N/L\unlhd E/L$ such that $(E/L)/(N/L)\cong
E/N\cong E_{\zeta}$ and such that $E/L$ is isomorphic to the subdirect
product of $E_\zeta$ and $E_\beta$ defined by identifying the common
quotient $H$. Let $\pi_N\colon E\to E/N$ be the natural projection. It also
follows from b) that  $\ker\pi_N = N$ and $\ker\pi_{\tilde D}= \{(1,1,v):
v\in V\}$ intersect trivially,  so  $\pi_N\times\pi_{\tilde D}\colon E\to E/N\times
\tilde D$ is injective. Since $E/N\cong E_\zeta$ and $\tilde D\cong D$, the
claim follows. 
\end{iprf}
\end{proof}

We can now  formulate the main result of this section:
 
\begin{theorem}
\label{thmhomogext}
Let $V$ be a simple $\Z_p H$-module and let $E$ be an extension of $H$ with a $V$-homogeneous module $K$. Let $L\unlhd E$ be the split kernel of $E$ (Definition \ref{defSK}). Then $S=E/L$ is a split extension of $H$ with a $V$-homogeneous module, and there exist an $n\in\mathbb{N}$ and  $\gamma_1,\ldots,\gamma_n\in Z^2(H,V)$ such that the cohomology classes in $H^2(H,V)$ induced by the $\gamma_i$ are all linearly independent and such that $E$ is the subdirect product of $S$ with $E_{\gamma_1},\ldots,E_{\gamma_n}$ (defined by identifying the common factor $H$).
\end{theorem} 
\begin{proof}
The statements about $S$ follow from Lemma~\ref{lemsubspl}. The kernel of the projection $S\to H$ is  $K/L$; the latter is $V$-homogeneous since it is a quotient of the $V$-homogeneous module $K$. The extension  $E$ can be considered as a subdirect product of extensions $E_{\beta}$, corresponding to cocycles
$\beta\in Z^2(H,V)$; let $\{\gamma_1,\ldots,\gamma_n\}$ be a minimal sub-multiset of $Z^2(H,V)$ such that $E$ is a subdirect product of $S$ with all those $E_{\gamma_i}$. We need to show that all the cohomology classes $\gamma_i+B^2(H,V)$ are linearly independent (which also shows that $\{\gamma_1,\ldots,\gamma_n\}$ is in fact a \emph{set} of size $n$). Assume the contrary, that is, without loss of generality  we can write
\[
\gamma_1=\sigma+\lambda_2\gamma_2+\ldots+\lambda_n\gamma_n
\]
for some $\lambda_i\in\Z_p$ and $\sigma\in B^2(H,V)$. An iterated application of
Lemma~\ref{lemSDP} (where $D$ is the subdirect product of $S$ with 
$E_{\gamma_2},\ldots,E_{\gamma_n}$  and $\gamma=\gamma_1$) shows that we can write $E$
as the subdirect product of $S$ with $E_{\sigma}$ and the $E_{\gamma_2},\ldots,E_{\gamma_n}$. Note that $E_\sigma$ is split, and Lemma~\ref{lemsubspl}c) shows that  the projection $E\to E_\sigma$ factors through $S$. We can therefore ignore $E_\sigma$ in the
construction and consider $E$ as the subdirect product of $S$ with
$E_{\gamma_2},\ldots,E_{\gamma_n}$, contradicting the minimality of $n$
\end{proof}

\begin{remark}\label{remProof}
In the proof of Theorem \ref{thmhomogext}, we could add redundant subdirect factors (stemming from linear
combinations of the $\gamma_i$) and, because we can choose which factor to eliminate, we can choose the cocycles $\gamma_1,\ldots,\gamma_n$ to correspond to an arbitrary basis of their span in $H^2(H,V)$.
\end{remark}

%%%%%%%%%%%%%%%%%%%%%%%%%%%%%%%%%%%%%%%%%%%%%%%%%%%%%%
\section{Construction of the $(V,e)$-cover}\label{secExt2} 

\noindent Let $V$ be a simple $\Z_p H$-module. The results of the previous section show that the cover $\hat H_{V,e}$ is
a subdirect product of a split part with non-split extensions. Recall that $\hat H_{V,e}=\hat H_{p,e}/V(\calmhep)$ by Definition \ref{defHomCompNew}, and Theorem \ref{thmPhi} describes $\hat H_{p,e}$ using a wreath product construction with $|H|$ factors $\mathbb{Z}_p$. We explain in Proposition \ref{propT} that the split part of $\hat H_{V,e}$ is covered by the image of a homomorphism $\Psi_{V,e}$ on the free group $F$. However, the definition of $\Psi_{V,e}$ passes  through $\calmhep$, which again is infeasible in practice. In Section \ref{secCT} we therefore provide an alternative construction that only  uses $H$ and $V$; this construction is based on the fact that the split part of $\hat H_{V,e}$ is covered by $H\ltimes (R_{H,p}/V(R_{H,p}))^e$, and $R_{H,p}/V(R_{H,p})\cong V^{r}$ is a cyclic $H$-module.

\subsection{A wreath product construction for the split case}

We reconsider the epimorphism 
\[
\Psi_p=\mu_1\times\cdots\times \mu_e\colon F \to (H\ltimes R_{H,p})^e.
\]
of \eqref{eqPsip} where $R_{H,p}\cong \Z_p H\cong \Z_p^m$ is the regular module of $H$ in
characteristic $p$. 
The proof of Theorem \ref{thmPhi} has shown that
\[
\hat H_{p,e}=\Psi_p(F)\quad\text{and}\quad \calm=\calmhep=\Psi_p(M).
\]
Each $\mu_j\colon F\to H\ltimes R_{H,p}$ maps the generator $x_k\in X$ of $F$ to
$(\psi(x_k),0)$ if $k\ne j$, and to $(\psi(x_j); 1)$ if $k=j$; here
$1=1(\ee)$ is the unit vector supported at the identity of $H$.
This vector also is a generator of the cyclic $H$-module $R_{H,p}$.

By Definition \ref{defRhp}, we have $R_{H,p}=D_1^{r_t}\oplus\ldots\oplus D_t^{r_t}$ with each $D_j/\rad(D_j)$ simple, and there is a unique index $i$ such that  
\[
D_i/\rad(D_i)\cong V;
\]
we fix $i$ and set $r=r_i$. Let $V(R_{H,p})$ be as in
Definition \ref{defHomCompNew}; then $R_{H,p}/V(R_{H,p})\cong V^r$ is the largest
$V$-homogeneous quotient of $R_{H,p}$ and $r$ is the dimension of an
absolutely simple summand of $V$ over the algebraic closure of $\Z_p$. Factoring out $V(R_{H,p})$, we get
homomorphisms $\mu_{V,j}\colon F\to H\ltimes V^r$ mapping $x_k$ to
$(\psi(x_k),0)$ if $k\ne j$, and to $(\psi(x_j); 1+V(R_{H,p}))$ if $k=j$; here $1+V(R_{H,p})$
is a generator of the cyclic module \begin{eqnarray}\label{eqVr}V^r=R_{H,p}/V(R_{H,p}).
\end{eqnarray} These maps can be combined to
\[
\Psi_{V,e}=\mu_{V,1}\times\ldots\times\mu_{V,e}\colon F\to (H\ltimes
V^r)^e.
\]
By definition $\ker\Psi_p\le \ker\Psi_{V,e}$, which implicitly
defines an epimorphism from $\hat H_{p,e}$ to $\Psi_{V,e}(F)$. This
epimorphism factors through $\hat H_{V,e}$, since $\Psi_{V,e}(F)$ is by
construction an extension of $H$ with a $V$-homogeneous module. Recall from Definition \ref{defHomCompNew} that $\hat H_{V,e}=\hat H_{p,e}/V(\calm)$. If  $\eta$ denotes the natural projection $\hat H_{p,e}\to\hat H_{V,e}$, then we get the following commutative diagram of successive projections

\begin{equation*}
\begin{tikzcd}
  F\arrow{r}{\Psi_p}\arrow[bend right=20,swap]{rrr}{\Psi_{V,e}} & \hat H_{p,e}\arrow{r}{\eta} & \hat H_{V,e}\arrow{r}{\theta} & \Psi_{V,e}(F)\arrow{r}{\pi} & H. 
\end{tikzcd}
\end{equation*}
We now prove that $\Psi_{V,e}(F)$ exhibits the split part of $\hat H_{V,e}$:
 
\begin{proposition}\label{propT}
Every $e$-generated split extension of $H$ with a $V$-homogeneous module is a quotient of $\Psi_{V,e}(F)$.
\end{proposition}
\begin{proof}
We use the notation introduced above the proposition. By Theorem~\ref{thmMain}, the representation module $\calm\leq \hat H_{p,e}$ is a direct sum 
$\calm=\mathcal{A}\oplus\mathcal{B}$ such that the $e$-generated
quotients of $\hat H_{p,e}$ which are split extensions of $H$ are exactly the quotients
that have $\mathcal{B}$ in the kernel. Thus, it remains to show that  $\hat H_{p,e}/V(\mathcal{A})\mathcal{B}$  is a quotient of $\Psi_{V,e}(F)$.  Recall from Remark \ref{remAB} that $\mathcal{A}$ is the direct sum of projective indecomposable modules
that are direct summands of the free module $R_{H,p}$; this implies that  $\mathcal{A}$ is projective itself, cf.~\cite[Definition~1.6.15]{luxpahlings}. By \cite[p.~18, Example~1.1.46]{luxpahlings}
the group algebra $\Z_p H\cong R_{H,p}$ is a symmetric algebra and for such algebras
every projective module is also
injective, see~\cite[Theorem~1.6.27(d)]{luxpahlings}. It follows
from $\mathcal{M}=\Psi_p(M)\leq (R_{H,p})^e$ that  $\mathcal{A}$ is  a submodule of $(R_{H,p})^e$, and therefore a direct summand by injectivity. This implies that
\[
V(\mathcal{A})=V((R_{H,p})^e)\cap\mathcal{A}.
\] 
This means that the image of the projection \[(H\ltimes R_{H,p})^e\to (H\ltimes R_{H,p}/V(R_{H,p}))^e=\Psi_{V,e}(F)\] exposes all the factors of $\mathcal{A}/V(\mathcal{A})$, which implies that  $\hat H_{p,e}/V(\mathcal{A})\mathcal{B}$  is a quotient of $\Psi_{V,e}(F)$.
\end{proof}

%%%%%%%%%%%%%%%%%

\subsection{A practical construction of $\Psi_{V,e}(F)$}\label{secCT}

The definition of $\Psi_{V,e}$ is on the free group $F$ and passes through $\calm$,
which we deemed infeasible in practice. We now provide an alternative,
synthetic,
description  that only uses $H$ and $V$. We denote the dimension of $V$ by $s$ and the multiplicity of $V$ in the radical factor $R_{H,p}/\rad(R_{H,p})=(D_1/\rad(D_1))^{r_1}\oplus\ldots\oplus (D_t/\rad(D_t))^{r_t}$ of $R_{H,p}$ by $r$. By Theorem~\ref{structureKH}, this
multiplicity is the dimension of an absolutely simple constituent $U$ of $V$, and $r$ divides $s$. 
 
As seen in \eqref{eqVr}, the $H$-module $V^r$ is isomorphic to a quotient of the
cyclic module $R_{H,p}$, so $V^r$ is cyclic as well. Suppose we have a
cyclic generator $z\in V^r$, then one can define  $H\ltimes V^r$ and
homomorphisms
\[ 
\psi_j'\colon F\to H\ltimes V^r
\] 
that map the generator $x_j\in X$ of $F$ to $(\psi(x_j),z)$ and $x_k\ne x_j$ to $(\psi(x_k),0)$.
It follows that, up to automorphisms, 
\begin{eqnarray}\label{eqPsiVe}
\Psi_{V,e}=\psi_1'\times\ldots\times \psi_e'\colon F\to (H\ltimes V^r)^e.
\end{eqnarray}
Thus, all that remains is to find a cyclic generator of $V^r$; we now describe how to do that.

Recall that here we have the field $\mathbb{F}=\mathbb{Z}_p$. As in
Theorem~\ref{structureKH}, let $\underline{\mathbb{F}}$ be a splitting field for $\mathbb{F}H$
and let $U$ be an absolutely simple $\underline{\mathbb{F}}H$-module that is a direct
summand of $\underline{\mathbb{F}}V$. We obtain $U$ from $V$ using MeatAxe~\cite{holtmeataxe} methods;
this also determines the value of $r$.

Let $\nu\colon\underline{\mathbb{F}}V\to U$ be the
natural projection onto that summand. We choose vectors $w_1,\ldots,w_r\in
V$ such that their images $\nu(w_1),\ldots,\nu(w_r)$ form an
$\underline{\mathbb{F}}$-basis of $U$. Since the images of the standard $\mathbb{F}$-basis of
$V$ span $U$ as an $\underline{\mathbb{F}}$-vector space, we can take $\{w_1,\ldots,w_r\}$
as a subset of such a standard basis. Since $U$ is absolutely simple, it
follows from \cite[Corollary~1.3.7]{luxpahlings} that $\underline{\mathbb{F}}H$ acts
as a full matrix algebra on $U$. This means that we can find elements
$a_i\in\underline{\mathbb{F}}H$ such that $\nu(w_i)^{a_j}=\delta_{i,j}\nu(w_i)$ for all $i,j$,
where $\delta_{i,j}$ is the Kronecker-delta. We now consider $U^r$ as a
quotient of $(\underline{\mathbb{F}}V)^r$ and let $\underline{\mathbb{F}}H$ act diagonally.
For $w\in U$, denote by $[w]_i$ the vector $w$ in the $i$-th component of
$U^r$, and define
\[
z=[\nu(w_1)]_1+[\nu(w_2)]_2+\cdots+[\nu(w_r)]_r\in U^r.
\]
By construction, each $z^{a_i}=[\nu(w_i)]_i$. Since $U$ is
simple, each  $\nu(w_i)$ generates $U$ as $\underline{\mathbb{F}}H$-module; this shows that  $z$ generates $U^r$ as $\underline{\mathbb{F}}H$-module. Since $V$ is a simple $\mathbb{F}H$-module, this implies that the pre-image
$[w_1]_1+\cdots+[w_r]_r\in V^r$ of $z$ generates $V^r$ as $\mathbb{F}H$-module. We have therefore found a cyclic generator.

%%%%%%%%%%
\subsection{A practical construction of $\hat H_{V,e}$}\label{secHve}

We combine the results of Theorem~\ref{thmhomogext} with the
construction in Section~\ref{secCT} and get the following construction of the epimorphism $\eta\circ\Psi_p\colon F\to \hat H_{V,e}$:

\begin{theorem}\label{thmHve}
Let $V$ be a simple $\Z_p H$-module. Let $F$ be the free group on
$\{x_1,\ldots,x_e\}$ with associated epimorphism $\psi\colon F\to H$. Let $\gamma_1,\ldots,\gamma_d\in Z^2(H,V)$ such that their images in $H^2(H,V)$
form a basis. For each $i$, let $E_i=E_{\gamma_i}$ with projections $\varepsilon_i\colon E_i\to H$, and let $\varrho_i\colon F\to E_i$ be defined by $\varrho_i(x_k)=(x_k,1)\in E_i$ for all $k$, that is, $\varepsilon_i(\varrho_i(x_k))=\psi(x_k)$. If we define
\[
\rho=\Psi_{V,e}\times\varrho_1\times\cdots\times\varrho_d
\colon F\to (H\ltimes V^r)^e\times E_1\times\cdots\times E_d
\]
with $\Psi_{V,e}$ as in \eqref{eqPsiVe}, then $\ker\rho=\ker(\eta\circ\Psi_p)$ and  $\rho(F)\cong \hat H_{V,e}$.
\end{theorem}
\begin{proof}
Since $F$ is free, each $\varrho_i$ is a homomorphism whose image covers
all of $E_i/V\cong H$. Since each $E_i$ is non-split and $V$ is simple, each $\varrho_i$ is surjective. The image of $\rho$ therefore is an extension of $H$ with a $V$-homogeneous
module, and therefore it is a quotient of $\hat H_{V,e}$. On the other hand, Theorem~\ref{thmhomogext} and Remark \ref{remProof} show that $\hat H_{V,e}$ is a
subdirect product of a split extension (which is $\Psi_{V,e}(F)$ by Proposition~\ref{propT}), with extensions corresponding to a basis of a subspace of 
$H^2(H,V)$. A basis of all $H^2(H,V)$ will suffice, which shows that
$\rho$ exposes all of $\hat H_{V,e}$. 
\end{proof}

The last ingredient that is required in order to construct $\hat H_{V,e}$ in practice is to be able to calculate $H^2(H,V)$ and to construct the extension associated to a particular
cocycle. A method for this has been given in~\cite{holtcohom}. Here we use an 
alternative approach, utilizing confluent rewriting systems; we will describe this method and its advantages in Section~\ref{secComp}.

%%%%%%%%%%%%%%%%%%%%%%%%%%%
\section{Quotient algorithm: lifting epimorphisms}\label{secLiftEpi}

\noindent As an application of the results established so far, we describe a
quotient algorithm that does not require the initial factor group to be
solvable. We assume that $G=F/R$ is a finitely presented group and that an
epimorphism $\varphi\colon G\to H$ onto a finite group is given.
(Section~\ref{secInitialQuot} below gives a sketch how such a homomorphism
$\varphi$ could be found.) We assume
that we can determine a confluent rewriting system for $H$, as well as a
faithful representation in characteristic $p$, or a permutation
representation. This is a reasonable assumption, because if we cannot
compute with $H$, then it seems unlikely that  $\varphi$ can be used to
deduce information about $G$.

Our goal is to extend $\varphi$ to an epimorphism $\tau\colon G\to \tilde H$
such that $\tilde H$ is an extension of $H$ with a semisimple $\Z_pH$-module and such that $\tau$ factors through $\varphi$. As discussed in 
the introduction, imposing the requirement of semisimplicity is not a
restriction, because any such extension with a solvable normal subgroup can
be built as an iterated extension with semisimple modules.

We first classify the irreducible $\Z_pH$-modules. Following~\cite[Section~7.5.5]{handbook}, we do so by starting with 
the composition factors of
a faithful $\Z_p$-representation of $H$ and then iteratively computing
composition factors of tensor products until no new factors arise; see also \cite{ples87} for a description for solvable $H$.

When lifting epimorphisms for a second time,
we do not need to recompute modules, as long as we work with the
same prime, as any normal subgroup of $p$-power order lies in the kernel of
any irreducible representation in characteristic $p$.
(The latter follows
because the set of fixed points of the normal $p$-subgroup is a non-trivial
submodule.) Since
semisimple modules are the direct sum of homogeneous modules, we
now iterate over the simple modules, and for each such module $V$, we construct
the group $\tilde H_V$ that is the largest extension of $H$ that is a quotient of $G$
and whose projection onto $H$ has a $V$-homogeneous kernel. As a quotient of $G$, this group $\tilde H_V$ will also be a quotient of $F$ and therefore a quotient of $\hat H_{V,e}$. Indeed, because $G$ is defined as a quotient of $F$ by a relator set $R$, we obtain $\tilde H_V$ (and the associated epimorphism)
as a factor of $\hat H_{V,e}$ by the normal closure of the
relators $R$ evaluated in the generators of $\hat H_{V,e}$.
The cover $\tilde H$ (and the epimorphism on $H$) then
will be the subdirect product of all these extensions. 

\subsection{Finding the initial homomorphism}
\label{secInitialQuot}

While it is not the main subject of this paper, we briefly sketch  how one can find candidates for the initial epimorphism
$\varphi\colon G\to H$.  Since our algorithm constructs extensions with
solvable groups, it is sufficient for $H$ to be Fitting-free.
Thus $H$ embeds in the automorphism group of its socle, and therefore is a
subdirect product of groups $Q$ satisfying   $T^n\le Q\le\Aut(T^n)\cong\Aut(T)\wr {\rm Sym}_n$ for some finite simple group $T$ and integer $n>0$. Given a choice of $n$ and $T$
(respectively, using the classification of finite simple groups, a choice of
$n$ and $\sz{T}$) we can find all such quotients, albeit at a cost that is
exponential in $n$ and $\sz{T}$. This will provide a choice of candidates
for epimorphisms $\varphi\colon G\to H$ to seed our algorithm with:

Using the low-index algorithm~\cite[\S5.6]{simsbook}, we first search for
subgroups of $G$ of index up to $n$. For each such subgroup $S$, we search
for homomorphisms $\tau\colon S\to\Aut(T)$ such that $T\le\tau(S)$;  the representation of $G$ induced by $\tau$ then exposes the desired quotient
$Q$, see \cite{hulpke01}.  By the proof of Schreier's Conjecture, $\Aut(T)/T$ is
solvable of derived length at most $3$. Thus the third (or less, depending on
$T$) derived subgroup of $S$ maps  onto $T$. It therefore remains to find such epimorphisms $\tau$ from (derived subgroups of) $S$ onto
$T$, respectively onto almost simple groups with socle $T$. A basic way of
doing this is a generic epimorphism search such as described in~\cite[Section~9.1.1]{handbook}.
For $T$ being a classical group with particular parameters, there are
algorithms that find epimorphisms, utilizing the underlying geometry, see \cite{jambor,jamborphd,ples09,bridsonetal19}.

%%%%%%%%%%%%%%%%%%%%%%%
\section{Computing Cohomology via rewriting systems}\label{secComp}
\noindent  To make the construction of $\hat H_{V,e}$ in Theorem \ref{thmHve}
concrete and effective, we need to be able to calculate 2-cohomology groups
and extensions for the given finite quotient $H$. 
A general method for this task has been described in~\cite{holtcohom}, which
finds a cohomology group as a subgroup of the cohomology for a
Sylow $p$-subgroup of $H$ (here $p$ is  the characteristic of $V$), and returns non-split
extensions through presentations. We introduce a different approach that assumes a confluent rewriting system
for the group $H$, but also returns a confluent rewriting system for the
resulting extensions, making it easier to find the structure of subgroups
given by generators. The method we shall employ is a natural generalisation of the method
used in the polycyclic case~\cite[Section~8.7.2]{handbook},
and already arises implicitly in~\cite{ples87}, in~\cite{holt89}, in
Groves~\cite{groves97}, as well as in \cite{schm08}. A brief description is
also given in an (unpublished) manuscript of Stein~\cite{stein}.
We describe this method in detail here, as we were not able to find a complete and rigorous
treatment in the literature.

In this section, as before, let $H$ be a finite group with $e$
generators\footnote{We use the same variable $e$ here,
although it is not necessary to use the same generating set as in the
quotient algorithm. The choice (and number) of
generators used for the cohomology calculation does not
need to agree with the
images of free generators used for the construction of $\hat H_{V,e}$; it is sufficient that we can translate between different generating systems.}
$\{h_1,\ldots,h_e\}$.
We shall also assume that we have rules for a confluent rewriting system for $H$ in these generators; see \cite[Chapter~12]{handbook} and \cite[Section~2.5]{simsbook} for
details on rewriting systems.  Such a rewriting system can be
composed from rewriting systems for the simple composition factors of $H$;
for non-abelian simple groups it can be found by using subgroups forming a
BN-pair (or similar structures)~\cite{schm11}.
Such a  rewriting system allows us to compute normal forms of elements in
$H$, given as words in the generators. In the following, $V$ is a $d$-dimensional $\Z_p H$-module
with $\Z_p$-basis $\{v_1,\ldots,v_d\}$.

\subsection{Extending the rewriting system} Starting with a confluent rewriting system for $H$ and the $\Z_pH$-module $V$, we explain how extensions of $H$ with $V$ can be described by  \emph{extending}  the original rewriting systems. This will lead to a method for computing  $H^2(H,V)$ via solving homogeneous linear equation systems. We first consider the quotient $H$, then the module $V$, and then the extensions.

\smallskip

\noindent {\bf The group  $H$.} By introducing formal inverses, $H$ can be considered as a monoid with
$2e$ monoid generators $\{h_1^{\pm 1},\ldots,h_e^{\pm 1}\}$.  The latter is a quotient of the free monoid $A$ on $\gesy{a}=\{a_1^{\pm
1},\ldots,a_e^{\pm 1}\}$, with natural epimorphism $\alpha\colon A\to H$ defined by $a_i^{\pm 1}\mapsto
h_i^{\pm 1}$. Note that $a_i^{-1}$ is a
formal symbol, while $h_i^{-1}$ is the inverse of $h_i$. Using a
Knuth-Bendix procedure \cite[Section~2.5]{simsbook}, we assume that we have
a  confluent rewriting system (with respect to a reduction order $\prec_a$)
for the monoid $H$ on this generating set. This rewriting system consists of
a set of rules $\calR_H$ each of the form $l\to r$
for certain words $l$ and $r$ in the generators of $A$, such that $r\prec_a
l$. Since we introduced extra generators to represent inverses, we assume that
$\calR_H$ contains rules that reflect this mutual inverse relation and  that
become trivial (or redundant) when considering the relations as group
relations: these are the rules of the form $a_ia_i^{-1}\to\emptyset$ and
$a_i^{-1}a_i\to\emptyset$, which we collect in a subset
$\overline\calR_H\subset\calR_H$; here $\emptyset$ denotes the empty word.
We note that this assumption holds automatically if $\prec_a$ is based on
length and all generators have order $2$. If the order of a generator $h_i$ is 2, then these rules will change shape: Without loss of generality, after possibly
switching $a_i$ and $a_i^{-1}$, the inversion rule becomes $a_i^{-1}\to a_i$, which we collect  in $\overline\calR_H$. The rule $a_i^2\to\emptyset$ (which must exist,
since otherwise $a_i^2$ cannot be reduced) however will not be part of
$\overline\calR_H$. We now set $\widetilde\calR_H=\calR_H-\overline\calR_H$,
so that our rules are partitioned as
\[
\calR_H=\overline\calR_H\cup\widetilde\calR_H.
\]

\noindent {\bf The module $V$.} We write the elements of the $\Z_p H$-module $V$ multiplicatively
as $\gesy{v}^{\gesy{e}}=v_1^{e_i}\cdots v_d^{e_d}$ with
$\gesy{e}=(e_1,\ldots e_d)\in\Z_p^e$. Let $\tau\colon H\to\Aut_{\Z_p}(V)$
describe the $\Z_p H$-action on $V$. Correspondingly, we choose an alphabet of
$d$ generators $\gesy{b}=(b_1,\ldots,b_d)$, and consider the set of rules
\begin{equation}
\calR_V=\{
b_i^p\to\emptyset,\;\;b_jb_j\to b_ib_j : i\in\{1,\ldots,d\}, j>i\}.
\label{defrb}
\end{equation}
These
rules form a reduced confluent rewriting system with respect to the ordering
$\prec_b$, which  is the iterated wreath product ordering of length-lex
orderings on words in a single symbol $b_i$. They define a normal form
$\gesy{b}^{\gesy{e}}=b_1^{e_1}b_2^{e_2}\cdots b_d^{e_d}$ with
$\gesy{e}\in\Z_p^d$. The set $\calR_V$ therefore describes a monoid  isomorphic
to $V$ via  $b_i\to v_i$. 

\smallskip

\noindent {\bf Extensions of $H$ with $V$.} We now take the combined alphabet 
$\calA=\{a_1^{\pm 1},\ldots,a_e^{\pm 1}\} \cup\{b_1,\ldots,b_d\}$ and denote by $\prec$ the 
wreath product ordering $\prec_b\wr\prec_a$, see \cite[p.\ 46]{simsbook}. We define  $\calR_M$ to be the set of all rules
\[
\calR_M=\{ b_ja_i^{\sigma}\to a_i^{\sigma}
\gesy{b}^{(f_{i,j,\sigma,1},\ldots,f_{i,j,\sigma,d})} : \sigma\in\{\pm1 \}, i\in\{1,\ldots,e\}, j\in\{1,\ldots,d\}\}
\]
where the exponents $f_{i,j,\sigma,k}$ are defined by
$v_j^{\tau(a_i^\sigma)}=\gesy{v}^{(f_{i,j,\sigma,1},\ldots,f_{i,j,\sigma,d})}$.

If $\widetilde\calR_H$ has $r$ rules, then corresponding to those  we define an ordered set of indeterminates over $\Z_p$, namely
\[
\gesy{x}=(x_{1,1},\ldots,x_{1,d},\; x_{2,1},\ldots,x_{2,d},\;\ldots,\; x_{r,1},\ldots,x_{r,d}),
\]
and define a set of new rules $\calR_H(\gesy{x})$ that consists of the rules in  $\widetilde\calR_H$ modified by a co-factor (or \emph{tail}), which is an element of $V$ given as a word in $\gesy{b}$ that is parameterized by the values
of the variables $\gesy{x}$:
\begin{equation}
\calR_H(\gesy{x})=\{ l_i\to r_i \gesy{b}^{(x_{i,1},\ldots,x_{i,d})} : (l_i\to r_i)\in \widetilde\calR_H\}.
\label{defra}
\end{equation} 
Lastly, we set
\[
\calR=\calR(\gesy{x})=\calR_H(\gesy{x})\cup\calR_V\cup\calR_M\cup\overline\calR_H.
\] 
In conclusion: the rules in $\calR_H(\gesy{x})$ (together with  $\overline \calR_H$) are the original rules of the rewriting system of the quotient $H$, with appended parametrised tails; the rules in $\calR_V$ encode the group structure of $V$, and the rules in $\calR_M$ encode its $H$-module structure. By the definition of the wreath product ordering, for all rules in $\calR$ we
have that the left hand side is larger than the right hand side; thus $\calR$ 
is a rewriting system. Since $\calR_V$ is always reduced, it follows that  $\calR$ is reduced if $\calR_H$ is.
 
We  aim to find conditions on the variables $\gesy{x}$ that make $\calR(\gesy{x})$ confluent, and first observe that in this case the
rewriting system describes a group extension as desired. We denote any particular
assignment of values to $\gesy{x}$ by $\gesy{y}\in\Z_p^{dr}$.

\begin{lemma}
For any $\gesy{y}\in\Z_p^{dr}$, the monoid presentation
$\langle\calA\mid\calR(\gesy{y})\rangle$ defines a group.
\end{lemma}
\begin{proof}
It is sufficient to show that every generator has an inverse.  The
rules $b_i^p\to\emptyset$ in (\ref{defrb}) show that every generator $b_i$
has an inverse.  As $H$ is a group, $\calR_H$ must contain rules that allow
for free cancellation. If the order of $h_i$ is not 2, these rules must be
of the form $a_ia_i^{-1}\to\emptyset$ and  $a_i^{-1}a_i\to\emptyset$. These
rules imply that $a_i$ and $a_i^{-1}$ are mutual inverses and they must lie
in $\overline\calR_H\subseteq \calR(\gesy{y})$.  If the order of $h_i$ is
$2$, then there will be a rule $a_i^{-1}\to a_i$ in $\overline\calR_H$ (thus
the generator $a_i^{-1}$ is a redundant, duplicate, generator) and a rule
$a_i^2\to\emptyset$ in $\widetilde\calR_H$; this  last rule implies by
(\ref{defra}) the existence of a rule $(a_i^2\to w)\in\calR_H(\gesy{y})\subset
\calR(\gesy{y})$, with $w$ a word in the generators $\{b_1,\ldots,b_d\}$
only. Thus $w$ represents an invertible element, and $a_i w^{-1}$ will be an
inverse for $a_i$.
\end{proof}

Thus we can consider $\calR(\gesy{y})$ as relations of a group
presentation with abstract  generators
\[
\calA'=\{a_1,\ldots,a_e,b_1,\ldots,b_d\};
\]
note that some of  the relations might
become vacuously true in a group.  Since $H$ acts linearly on $V$, the set of values of $\gesy{x}$ that make the rewriting system confluent is a subspace of $\Z_p^{dr}$,  denoted by
\begin{eqnarray}\label{eqConf}
X=\{\gesy{y}\in \Z_p^{dr}: \calR(\gesy{y})\text{ confluent}\}.
\end{eqnarray}

\begin{lemma} 
\label{lemAlwaysExt}
If $\gesy{y}\in X$, then $\langle\calA'\mid\calR(\gesy{y})\rangle$
defines a group that is an extension of $H$ with $V$ where the conjugation
action of $H$ equals the module action.
\end{lemma}
\begin{proof}
The relations in $\calR_V$ and $\calR_M$ show that  $N=\langle
b_1,\ldots,b_d\rangle$ is abelian and  normal. As the only relations in
$\calR$ whose left side only involves the generators $b_1,\ldots,b_d$ are
the relations in $\calR_V$, confluence of $\calR$ implies that no other
rules apply to a word in these generators, thus
$N$ is isomorphic to $V$.  The factor group can be described by setting all
$b_i$ to 1 in the relations; this produces the rules $\calR_H$, and those
define $H$. The rules in $\calR_M$ prove the claim about the action.
\end{proof}

\noindent {\bf Cohomology.} Vice versa, consider an extension $E$ of $H$ with  $V$ defined by $\gamma\in Z^2(H,V)$. Note that $E$ has underlying set $H\times V$ with multiplication
$(g,v)(h,w)=(gh,v^hw\gamma(g,h))$, see \eqref{extensionmult}. For the chosen generators $h_i$ of $H$,
corresponding to the rewriting system $\calR_H$, we set $u_i=(h_i,1)$, and
let $\gesy{u}=(u_1,\ldots,u_e)$. We also choose a basis $\gesy{v}$ for (the
image in $E$ of) $V$. The elements in $\gesy{u}\cup\gesy{v}$ satisfy
the relations in $\calR_V\cup\calR_M\cup\overline\calR_H$.
Furthermore, for any rule $l_i\to r_i \gesy{b}^{(x_{i,1},\ldots,x_{i,d})}$
in $\calR_H(\gesy{x})$, we can find an assignment for the $\{x_{i,j}\}_j$ to
values in $\Z_p$, such that this rule evaluated at $\gesy{u}\cup\gesy{v}$
holds. Thus there exists $\gesy{y}\in \Z_p^{rd}$, such that the rules in
$\calR(\gesy{y})$ hold in $E$. Since these rules imply a normal form for the $H$-part and for the $V$-part,
we  know that this rewriting system is confluent, that is, 
$\gesy{y}\in X$. Because of Lemma \ref{lemAlwaysExt}, this process defines a  surjective map
\[
\xi\colon Z^2(H,V)\to X,
\]
such that $E_\gamma$ is isomorphic to the group $\langle \mathcal{A}'\mid \calR(\xi(\gamma))\rangle$ determined by $\xi(\gamma)\in\Z_p^{dr}$. By construction, the $dr$ entries in $\xi(\gamma)$ are products of elements of the form $\gamma(a,b)^c$ with $a,b,c\in H$, so  $\xi$ is a linear map. Finally, if $\gamma\in\ker \xi$, then $\xi(\gamma)=\gesy{0}$ and the group
given by $\calR(\gesy{0})$ is a split extension (as the elements
representing $H$ form a subgroup), thus $\gamma\in B^2(H,V)$. We summarize:
 
\begin{theorem} 
\label{thmcohom}
The tuples $\gesy{x}$ that make $\calR$ confluent form a $\Z_p$-vector space
$X=\xi(Z^2(H,V))$ and $\ker\xi\leq B^2(H,V)$, hence  $H^2(H,V)\cong \xi(Z^2(H,V))/\xi(B^2(H,V))$.
\end{theorem}
  
\subsection{Making the system confluent}
We now describe how to compute the images of $Z^2(H,V)$ and $B^2(H,V)$ under $\xi$, leading to a construction of $H^2(H,V)$ via Theorem \ref{thmcohom}.

We start with  $Z^2(H,V)$ and recall that $\xi(Z^2(H,V))=X$ as in \eqref{eqConf}. Using the Knuth-Bendix
method as described in \cite[Section 2.3]{simsbook}, to compute $X$ we need to consider
overlaps of left hand sides of rules in $\calR(\gesy{x})$. Set $\overline\calR_H(\gesy{x})=\calR_H(\gesy{x})\cup\overline\calR_H$, so \[\calR(\gesy{x})=\overline\calR_H\cup\calR_M\cup\calR_V\] is the union of three
sets. Thus, there will be six kinds of overlaps, which we now consider
separately. Overlaps of left hand sides of rules in $\calR_V$ reduce uniquely by the
definition of $\calR_V$. The left hand sides of two rules in $\calR_M$ cannot overlap because of
their specific form. Similarly, rules in $\calR_V$ and $\overline\calR_H$ cannot overlap as their
left hand sides are on disjoint alphabets. The overlap of a left hand side in $\calR_V$ and in $\calR_M$ will have the 
form $w_{\gesy{b}}a_i^{\pm 1}$ (where $w_{\gesy{b}}$ is a word expression in
$\gesy{b}$) and reduces uniquely as the action on $V$ is linear. A left hand side in $\calR_M$ and one in $\overline\calR_H$ will overlap in
the form $w_{\gesy{b}}w_{\gesy{a}}$; such expressions reduce uniquely as the
action on a module is a group action. This leaves overlaps of left hand sides in $\overline\calR_H$.  For this we note that the rules in $\calR_V\cup\calR_M$ allow us to
transform any word expression into a form $w=ab$ (called \emph{clean}) where
$a$ is a word in $\gesy{a}$, and $b$ a reduced word in $\gesy{b}$. We call
these factors the $\gesy{a}$-part and $\gesy{b}$-part, respectively.
Furthermore, the $\gesy{a}$-part of the clean form of a word is simply the
image of the word when setting all generators in $\gesy{b}$ to one.
As every rule in $\overline\calR_H$ corresponds to a rule in
$\calR_H$, and since $\calR_H$ is confluent, this together 
shows that the $\gesy{a}$-part of any reduced word will be unique. If we write a word as a product (in arbitrary order) of elements in
$\gesy{a}$ with powers of generators in $\gesy{b}$, the $\gesy{b}$-part
of a clean form of a word will be a normal form
$\gesy{b}^{(e_1,\ldots,e_d)}$, where the $e_i$ are homogeneous linear
functions in the exponents of the $\gesy{b}$-generators in the original
word. Reduction with rules in $\overline\calR_H$ will introduce powers of
$\gesy{b}$ with exponents given by variables in $\gesy{x}$. By reducing the overlap of two left hand sides of rules in $\overline\calR_H$,
and by reducing the resulting two words further to (arbitrary) reduced forms, we obtain clean words with equal $\gesy{a}$-parts and whose
$\gesy{b}$-parts are in normal form 
$\gesy{b}^{(e_1,\ldots,e_d)}$, where the $e_i$ are homogeneous linear
expressions in the variables $\gesy{x}$.

In conclusion, we have shown that the equality of the reduced forms of an overlap is  equivalent to a homogeneous linear equation in $\gesy{x}$; by processing all overlaps, we obtain a
homogeneous system of linear equations. Confluence of
$\calR(\gesy{x})$ for a particular set of values of $\gesy{x}$ then is
equivalent to $\gesy{x}$ satisfying this system over $\Z_p$; this allows us to compute $X=\xi(Z^2(H,V))$ as the solution space of a homogeneous linear equation system.

We can calculate  $\xi(B^2(H,V))$ in a way similar to the
establishment of the equations. For a function $\lambda\colon H\to V$ we
replace $a_i$ by $a_i\lambda(a_i)$ in the rules in $\calR_H(\gesy{0})$,
and use the rules in $\calR_M$ and $\calR_V$ to bring left and right side
into a clean form. Comparison of the remaining $\gesy{b}$-parts gives exponent
vectors that combine to the image of the associated cocycle under $\xi$.

%%%%%%%%
\section{Practical aspects}\label{secPI}

\subsection{Hybrid groups}\label{sechybgroup}

\noindent We comment on the new data structure we have introduced  to make
our algorithm more efficient. Recall that once the respective module(s) $V$
are chosen, our process to construct $\hat H_{V,e}$ builds on
algorithms to perform the following calculations:

\begin{items}
\item[(1)] Calculate $H^2(H,V)$ and construct extensions for particular cocycles.
\item[(2)] Construct semidirect products of $H$ with elementary abelian subgroups $V^r$.
\item[(3)] Construct direct products of the groups computed in Step (2).
\item[(4)] Construct subgroups of the groups computed in Step (3) that map onto the full factor group $H$.
\item[(5)] Form factor groups of the groups  computed in Step (4), by factoring out evaluated relators.
\end{items}
In most of these constructions, the result will always be a group that is the
extension of $H$ with an elementary abelian subgroup. We represent such groups as formal polycyclic-by-finite extensions, given as
a finitely presented group. Contrary to the more general construction
in~\cite{sinananholt}, we form a confluent rewriting system for the whole
group, which is used to calculate normal forms. Such a rewriting system can
be combined easily from a rewriting system for $H$ (which we anyhow have for
the purposes of computing the cohomology group), a polycyclic generating set for
the normal subgroup, and cocycle information that describes the extension
structure.  We call such a computer representation a {\em hybrid group}.
We also assume that we are able to translate between the generators for $H$ arising
as image of the generators of $F$, and the generators of the rewriting
system.

In practice, we split the rewriting system for a hybrid group $E$ into a
rewriting system for the non-solvable factor $H=E/N$, a polycyclic generating set
for the normal subgroup $N$, automorphisms of $N$ that represent the action
of factor group generator representatives, and cofactors (in $N$) associated
to the rewriting rules for the factor. Arithmetic in $E$ then uses the
built-in arithmetic for polycyclic elements, as this will be faster than an
alternative rewriting implementation. Indeed, if $H$ has a solvable normal
subgroup, arithmetic will be faster if, in a given hybrid group, we modify
the extension structure to have the solvable normal subgroup as large as
possible.

As for the algorithmic requirements listed above, the
information available from the cohomology computation is exactly what is
needed to represent extensions as hybrid groups. The construction of (sub)direct products or semidirect products is similarly
immediate. For a subgroup $S$ of a hybrid group, given by generators, such
that $SN=E$ (this holds for all subgroups we encounter), we can calculate generators
for $S\cap N$ from the presentation for the factor group, and then determine an
induced polycyclic generating set for $S\cap N$. This allows us to represent $S$ by its own hybrid representation. In the same way, a polycyclic generating set for factor groups can be used
to represent factors by normal subgroups contained in $N$. All calculations of the quotient algorithm therefore can take place in hybrid
groups, all for the same factor $H$. Since the order of these groups is known, and since a rewriting system is a special case of a presentation, we could use representations induced by the
abelianization of subgroups (as suggested in~\cite{hulpke01}) to find
faithful permutation representations.

It clearly would be of interest to study the feasibility of these
hybrid groups for general calculations. Doing so will require significant
more infrastructure work for these groups than we have currently done. While
we are optimistic about the general practical feasibility of such a
representation (e.g.\ following~\cite{sinananholt}), we do not want to make
any such claim at this point.

%%%%%%

\subsection{Cost estimates}\label{secCost}
It seems difficult to obtain  complexity statements that reflect practical behaviour. For example, even proving that computing with polycyclic groups has a favourable complexity is difficult
because of the challenges involving \emph{collection}, see \cite{NNcomp}. Despite these obstacles, it is still clearly beneficial  to be able to study a finitely presented group via a polycyclic quotient. The algorithmic framework considered in this work faces similar obstacles. Nevertheless, below we briefly discuss some cost estimates of some of the tasks required for the construction of $\hat H_{V,e}$.

Following~\cite{schm11}, obtaining a confluent rewriting system for $H$
essentially means to determine a composition series of $H$, and to look up
precomputed rewriting systems for the simple composition factors; the system
will asymptotically have $r \le\sqrt{\sz{H}}$ rules, though in many cases this
bound is far from reality. Determining $H^2(H,V)$ then requires solving a linear
system with $r\dim(V)$ variables and $r^2$ equations.  If $H$ is simple with
BN-pair, then the maximal length $\ell$ of a word in normal form for this
rewriting system (created in \cite{schm11}) is bounded by $\Oof(\log(\sz{H}))$, but it could be as large as $\sz{H}/2$ if $H$ is cyclic of
prime order. Assuming $\sz{H}$ has only small prime divisors, we get
$\ell=\Oof(\log\sz{H})$.

We now estimate the cost of multiplication in a hybrid group
$E$ with $E/N=H$ and $N$ abelian. Calculating the image of an element in $N$
under a word (of length up to $\ell$) representing an element of $H$ requires taking $\ell$ images of elements of $N$ under
homomorphisms, and  each such image requires $\log{|N|}$ multiplications in $N$. Considering elements in $E$ as pairs, the
first step of multiplying $h_1\cdot n_1$ and $h_2\cdot n_2$  in $E$ is to compute $n_1^{h_2}n_2$, at the cost of
$\ell\log{|N|}+1$ multiplications in $N$. Computing the product $h_1\cdot h_2$ then involves a
reduction sequence, say of length up to $s$, using the rewriting system for
$H$. Applying such an extended rewriting rule, say $w\to u\cdot n$ with (potential)
tail $n$, to a word $a\cdot w\cdot b$ results in $a\cdot u\cdot b\cdot n^b$ and  requires another homomorphic image computation. Multiplication in $E$ therefore requires up to $(s+1)\ell\log{|N|}$  products in $N$.

For constructing $\hat H_{V,e}$ via Theorem \ref{thmHve}, we form (for
$\Psi_{V,e}(F)$ and the extensions $E_i$) an extension of $H$ with
$e\dim(V)+\dim(H^2(H,V))$ copies of $V$. Even if the cohomology group is small, we work in an extension with a normal subgroup of order 
$\approx p^{\dim(V)^2}$, so $\log(\sz{N})\sim \dim(V)^2$. The cost of lifting an epimorphism $\varphi\colon G\to H$ to $\tau\colon G\to\tilde H$ with a
maximal $V$-homogeneous kernel is therefore proportional to
$v(s+1)\ell\dim(V)^2$, where $v$ is the sum of the lengths of the relators
defining $G$. 

In practical calculations, the main bottleneck for the
algorithm currently lies in the application of rewriting rules. At the
moment, this is done by a generic rewriting routine, operating on words. This could clearly be improved, for example by moving code from the system
library into the kernel, and by changing the order in which rules are
applied, in particular for cases with large elementary abelian subgroups.
Doing so, however is a substantial task on its own.

%%%%%%%%

\subsection{Example computations}\label{secEx}

\noindent As a proof of concept and to illustrate the capabilities of our methods, we have implemented
the algorithms described here  in the computer algebra system {\sf
GAP}~\cite{GAP}. The implementation of the 2-cohomology group and the
construction of extensions will be available with release 4.11. Our code for
hybrid groups, the construction of $\hat H_{V,e}$, and for lifting of
epimorphisms is available at \url{github.com/hulpke/hybrid}. We illustrate the scope of the algorithm and the performance of its
implementation in a number of examples; the code for those examples can be
found in the file \texttt{example.g} in the same GitHub repository. 
Calculation times are in seconds on a 3.7GHz 2013 Mac Pro with 16GB of memory available. We write extensions as $A.B.C=A.(B.C)$, etc.

The examples we consider here are all not solvable. While our
implementation also works for solvable groups, it becomes non-competitive in comparison to a dedicated solvable quotient
implementation: The reason for this is, at least in part, that element
arithmetic in the constructed covers, as well as the calculation of
cohomology groups, both go through a generic rewriting system  in the
routines library, instead of using dedicated kernel routines for groups with
a polycyclic presentation.

\begin{example}
\label{exHeineken}
The \emph{Heineken group} $\mathcal{H}=\left\langle
a,b,c\mid [a,[a,b]]=c,[b,[b,c]]=a,[c,[c,a]]=b
\right\rangle$ is infinite and 2-generated.
By von Dyck's Theorem
\cite[Theorem~2.53]{handbook}, there is, up to automorphisms,
a unique epimorphism $\varphi\colon
\mathcal{H}\to H$ onto the alternating group $H=A_5$, defined by
$\varphi(a)=(1,2,4,5,3)$ and $\varphi(b)=(1,2,3,4,5)$. It has been  shown in
\cite[p.\ 725]{hr99} that the largest finite nilpotent quotient of
$\ker\varphi$ has order $2^{24}$.
We now apply our algorithm: $A_5$ has three irreducible modules over $\Z_2$.
The trivial module yields a cover $2^3.A_5$ and lifts $\varphi$ to a
quotient of type $2.A_5$.  The absolutely irreducible module of dimension 4
yields a cover $2^{4\cdot 4}.A_5$ (we write $p^{a\cdot b}$ for a $b$-fold
direct product of an $a$-dimensional module) and lifts $\varphi$ to a
quotient $2^4.A_5$, and the other module of dimension 4 yields a cover
$2^4.A_5$ that does not lift $\varphi$. Table~\ref{tabitercox} shows results
and timings when iterating the lifting process until the maximal quotient
for prime $p=2$ has been found and confirmed as maximal. The whole
calculation took about $4\sfrac{1}{2}$ minutes on a 3.7GHz 2013 Mac Pro with 16GB of memory available.
\end{example} 

\begin{example}
The group $G=G_{(3,4,15;2)}=\left\langle a,b\mid a^3,b^4,(ab)^{15},[a,b]^2\right\rangle$ is an example of presentations of type ``$(m,n,p;q)$'' going back to Coxeter, and it is
known that $G$ is infinite, see \cite{thom95}. The group $G$ has a unique quotient isomorphic
to $A_6$. For characteristic 3 we obtained the quotients in Table~\ref{tabitercox}.
\end{example}

\begin{example}
The group $G=G_{(3,7,15;10)}=\left\langle a,b\mid
a^3,b^7,(ab)^{10},[a,b]^{10}\right\rangle$ is also known to
be infinite. It has four different epimorphisms $\varphi_i\colon G\to A_{10}$, distinguished by having different kernels. Contrary to the previous two examples, it is hard to find usable presentations
for the corresponding kernels, as the index $10!/2$ is large. This
example is therefore intractable with traditional methods.
Here we only
considered the modules of small dimensions, as the next smallest
dimension would be $26$, resulting in the construction of an (abelian)
polycyclic group with $2\cdot 26^2=1352$ generators. Working with
automorphisms of such a group would end up being unreasonably slow, because {\sf GAP} currently has no special treatment of abelian polycyclic groups. The given runtimes also exclude the cost of determining the irreducible
modules.  For $\varphi_1$, the algorithm finds a lift to a group
$(2\times 2^{8\cdot 3}).A_{10}$ in 112 seconds. Lifting again produces a larger quotient $2^{1\cdot
5}.(2\times 2^{8\cdot 3}).A_{10}$ in 774 seconds.
In characteristic 3, we find a quotient $(3\time 3^{9\cdot 2}).A_{10}$, in
characteristic $5$ a quotient $5^{8\cdot 1}.A_{10}$. For $\varphi_2$ and $\varphi_4$, we find a quotient of type $2^{8\cdot 1}.A_{10}$, for
$\varphi_3$ a quotient of type $(2\times 2^{8\times 1}).A_{10}$, thus
showing that these $A_{10}$ quotients fall in at least three different
equivalence classes.
\end{example}

\vspace*{-2ex}

\renewcommand\arraystretch{1.1}
{\small\begin{table}[htb]
 \begin{tabular}[t]{r|r}
   Isomorphism type of quotient&Time\\\hline
$2.(2\times 2^4).A_5$&1\\
$2^4.2.(2\times 2^4).A_5$&2\\
$2^4.2^4.2.(2\times 2^4).A_5$&5\\
$(2\times 2).2^4.2^4.2.(2\times 2^4).A_5$&11\\
$2^4.(2\times 2).2^4.2^4.2.(2\times 2^4).A_5$&23\\
$2^4.2^4.(2\times 2).2^4.2^4.2.(2\times 2^4).A_5$&73\\
No larger quotient for $p=2$&140
\end{tabular}\quad\quad\quad\begin{tabular}[t]{r|r}
Isomorphism type of quotient&Time\\\hline
$(3\times 3^6).A_6$&7\\
$3^{4\cdot 2}.(3\times 3^6).A_6$&30\\
$(3^4\times 3^{6\cdot 2}\times 3^9).3^{4\cdot 2}.(3\times 3^6).A_6$&473
\end{tabular}\\[1ex]
\caption{Isomorphism types of the iterated quotients of the Heineken group (left) and $G_{(3,4,15;2)}$ (right); computations were carried out on a 3.7GHz 2013 Mac Pro with 16GB of memory available; times are given in seconds.}
\label{tabitercox}
\end{table}}
  
\begin{example}
To illustrate the behaviour with larger quotients, we consider prime $2$ and the group
\[G=\left\langle a,b\mid a^3,b^6,(ab)^6,(a^{-1}b)^6\right\rangle,\]which is example $P_{10}$ in~\cite{niem94}; this group has a quotient of isomorphism
type $A_7$. The maximal lift of this quotient with an elementary abelian
kernel is 
\[
(2\times 2^{4\cdot2}\times 2^{4\cdot2}\times 2^{14\cdot 3}\times
2^{20\cdot 7}).A_7,
\]
and is found in about 71 minutes on a 3.7GHz 2013 Mac Pro with 16GB of memory available. If we restrict to simple modules of dimension $<5$, then we find a lift to
$(2\times 2^{4\cdot2}\times 2^{4\cdot2}).A_7$ in 2 seconds. Under the same restrictions, we can lift this to
$
(2^{1\cdot 5}\times 2^{4\cdot 2}\times 2^{4\cdot 2}).%
(2\times 2^{4\cdot2}\times 2^{4\cdot2}).A_7$,
in 56 seconds. The third lift to
\[
(2^{1\cdot 6}\times 2^{4\cdot 2}\times 2^{4\cdot 2}).%
(2^{1\cdot 5}\times 2^{4\cdot 2}\times 2^{4\cdot 2}).%
(2\times 2^{4\cdot2}\times 2^{4\cdot2}).A_7,
\]
is found in 435 seconds, and the fourth lift to a quotient of size $2^{84}\cdot |A_7|$ is found after about 2 hours:
\[
(2^{1\cdot 8}\times 2^{4\cdot 2}\times 2^{4\cdot 2}).%
(2^{1\cdot 6}\times 2^{4\cdot 2}\times 2^{4\cdot 2}).%
(2^{1\cdot 5}\times 2^{4\cdot 2}\times 2^{4\cdot 2}).%
(2\times 2^{4\cdot2}\times 2^{4\cdot2}).A_7.
\]
\end{example}

\section*{Acknowledgments}
\noindent Both authors thank the anonymous referee for their very careful reading. Dietrich was supported by an Australian Research Council grant,
identifier DP190100317. Hulpke was supported in part by National Science
Foundation grant DMS-1720146. Parts of the research were performed during visits of the second author at
Hausdorff Institute in Bonn, Germany in autumn 2018 and at Monash University in
Melbourne, Australia in autumn 2019. The hospitality of both institutions
is gratefully acknowledged.

\appendix

\section{Proofs of some representation theory results}\label{app} 

\noindent For the sake of completeness, we provide the proofs  missing in Section \ref{secDefR}.
 
\begin{proof}[Proof of Lemma \ref{lemRad}]
  \begin{iprf}
  \item If $D<A$ is a maximal submodule, then $C/(C\cap D)$ embeds in the simple module $A/D$, so $C\cap D=C$ or $C\cap D< C$ is maximal. In both cases, $\rad(C)\leq C\cap D$, so $\rad(C)\leq \rad(A)$. This also shows $\rad(A)\oplus\rad(B)\leq \rad(A\oplus B)$.  Conversely, if $W<A$ and $V<B$ are maximal, then $W\oplus B, A\oplus V<A\oplus B$ are maximal, so $\rad(A\oplus B)\leq \rad(A)\oplus \rad(B)$.
  \item Let $D=\sigma(A)$, so $\sigma\colon A\to D$ is surjective and $\rad(D)\leq \rad(B)$ by a).  If $V<D$ is maximal and $W$ is the full preimage of $V$ under $\sigma$, then $\sigma$ induces an isomorphism  $A/W\cong D/V$, and  $W<A$ is maximal. Thus, $\rad(A)\leq W$, and so $\sigma(\rad(A))\leq V$. Thus, $\sigma(\rad(A))\leq \rad(D)$.
  \item The module $A/\rad(A)$ embeds into a direct sum of simple modules, hence is semisimple. If $\rad(A)\leq C$, then $A/C \cong (A/\rad(A))/(C/\rad(A))$ is semisimple. Conversely, if $A/C$ is semisimple, then $\rad(A/C)=0$ and so $\rad(A)\leq C$. To prove the first claim, let $E\leq A$ be the submodule with $E/C=\rad(A/C)$. Applying  b) to the  projection $A\to A/C$ yields $(\rad(A)+C)/C \leq \rad(A/C)$, so $\rad(A)+C\leq E$. Now $B=\rad(A)+C$ is a submodule of $A$ with $(A/C)/(B/C)\cong A/(C+\rad(A))$ semisimple. Thus, $\rad(A/C)\leq B/C$, and so $E=\rad(A)+C$.
  \end{iprf} 
\end{proof}

\begin{proof}[Proof of Lemma \ref{lemKH}] 
Write $N= \mathbb{F}H$. Being finite fields, 
$\underline{\mathbb{F}}\geq \mathbb{F}$ is a Galois extension, so
\cite[Theorem~1.8.4]{luxpahlings} proves a). If $S<N$ is a maximal
submodule, then  $\underline{\mathbb{F}}(N/S)\cong \underline{\mathbb{F}}
H/\underline{\mathbb{F}}S$ is semisimple, so
$\rad(\underline{\mathbb{F}} H)\leq \underline{\mathbb{F}}\rad(N)$.
Conversely, $\underline{\mathbb{F}} H/\underline{\mathbb{F}}\rad(N)\cong \underline{\mathbb{F}}(N/\rad(N))$ and  $N/\rad(N)$ is a direct sum of
simple $\mathbb{F}H$-modules; now a) shows that
$\underline{\mathbb{F}}(N/\rad(N))$ is a semisimple
$\underline{\mathbb{F}}H$-module. This implies
$\rad(\underline{\mathbb{F}}H)\leq \underline{\mathbb{F}}\rad(N)$, and
therefore equality is established. This proves b).
\end{proof}

 \begin{proof}[Proof of Theorem \ref{structureKH}]
Most of this  follows from the  Krull-Schmidt
Theorem~\cite[Theorem~1.6.6]{luxpahlings} and Remark 1.6.22(a), Theorem 1.6.24, Theorem 1.6.20(b) in \cite{luxpahlings}. It remains to provide a proof for the multiplicity $r_i$ in the case that $\mathbb{F}$ is not algebraically closed. Let $\underline{\mathbb{F}} D_i=C_1\oplus\cdots\oplus C_k$ be a direct sum of $\underline{\mathbb{F}}$-projective indecomposables; note that the $C_j$ are direct summands of the regular module $\underline{\mathbb{F}}H$. By Lemma \ref{lemKH}, each $C_j/\rad(C_j)$ is a simple $\underline{\mathbb{F}}H$-modules and the isomorphism type of $C_j$ is determined by the isomorphism type of $C_j/\rad(C_j)$; in particular, the direct sum constituents of $\underline{\mathbb{F}} (D_i/\rad(D_i))$ are the the simple factors  $C_j/\rad(C_j)$ and they are
mutually non-isomorphic. The multiplicity $r_i$ of $D_i$ as a direct summand of $\mathbb{F}H$ thus equals the multiplicity of $C_j$ as a direct summand of $\underline{\mathbb{F}}H$, which is the multiplicity of $C_j/\rad(C_j)$ as a direct summand of $\underline{\mathbb{F}}H/\rad(\underline{\mathbb{F}}H)$. Wedderburn's
theorem implies $r_i=\dim_{\overline{\mathbb{F}}} (C_j/\rad(C_j))$, see \cite[Remark~1.6.22(a), Theorem~1.6.24]{luxpahlings}. 
\end{proof}

{\small
\bibliographystyle{plain}

}

\end{document}